\documentclass[12pt]{amsart}
\usepackage{geometry}

\usepackage[mathscr]{euscript}
\usepackage[all]{xy}
\usepackage[T1]{fontenc}
\usepackage{amsfonts}
\usepackage[small,nohug,heads=littlevee]{diagrams}
\usepackage{enumitem}
\setlist[itemize]{align=parleft,left=10pt..1em}

\usepackage{subcaption}
\usepackage{amssymb}
\usepackage{mathrsfs}
\usepackage[scr=rsfs,cal=boondox]{mathalfa}
\usepackage{soul}
\setstcolor{red}

\usepackage[dvipsnames,usenames]{xcolor}
\usepackage[colorlinks=true,urlcolor=blue,linkcolor=blue,citecolor=blue]{hyperref}
\usepackage{upgreek}
\hypersetup{
	linkbordercolor={1 0 0}, 
	citebordercolor={0 1 0} 
}
\usepackage{cite}
\usepackage[noabbrev]{cleveref}

\usepackage{tikz}

\newcommand{\al}{\alpha}
\newcommand{\be}{\beta}
\newcommand{\ga}{\gamma}

\renewcommand{\th}{\theta}
\newcommand{\la}{\lambda}
\newcommand{\si}{\sigma}

\newcommand{\om}{\omega}

\newcommand{\La}{\Lambda}

\newcommand{\De}{\Delta}
\newcommand{\Si}{\Sigma}
\newcommand{\ZZ}{{\mathbb Z}}
\newcommand{\CC}{{\mathbb C}}
\newcommand{\QQ}{{\mathbb Q}}
\newcommand{\RR}{{\mathbb R}}

\newcommand{\FF}{{\mathbb F}}

\newcommand{\cG}{\mathcal G}

\newcommand{\cL}{\mathcal L}

\newcommand{\tr}{\mathsf{T}}
\newcommand{\lto}{\longrightarrow}
\newcommand{\lk}{\operatorname{\ell{\it k}}}
\newcommand{\vlk}{\operatorname{{\it v}\ell{\it k}}}
\newcommand{\sig}{\operatorname{sig}}
\newcommand{\sm}{\smallsetminus}
\newcommand{\co}{\colon}

\newcommand{\Int}{{\text{Int}}}
\newcommand{\wh}{\widehat}
\newcommand{\wt}{\widetilde}

\newcommand{\ind}{\operatorname{ind}}

\newcommand{\Ker}{\operatorname{Ker}}

\newcommand{\cC}{\mathscr{C}}
\newcommand{\mG}{\mathcal{m} \hspace{-0.3mm}\mathcal{G}}
\newcommand{\vC}{\mathcal{v}\hspace{-0.25mm} \mathcal{C}}
\newcommand{\vG}{\mathcal{v}\hspace{-0.25mm} \mathcal{G}}

\newcommand*\wbar[1]{
  \hbox{ \kern-0.2em%
    \vbox{%
      \hrule height 0.75pt  
      \kern0.25ex
      \hbox{%
        \kern-0.10em
        \ensuremath{#1}%
        \kern-0.05em
      }%
    }%
  \kern0.05em}%
}
\makeatletter
\newcommand*\bigcdot{\mathpalette\bigcdot@{.55}}
\newcommand*\bigcdot@[2]{\mathbin{\vcenter{\hbox{\scalebox{#2}{$\m@th#1\bullet$}}}}}
\makeatother

\newtheorem{theorem}{Theorem} [section]
\newtheorem{lemma}[theorem]{Lemma}
\newtheorem{proposition}[theorem]{Proposition}

\newtheorem{corollary}[theorem]{Corollary}

\theoremstyle{definition}     
\newtheorem{definition}[theorem]{Definition}

\theoremstyle{remark}
\newtheorem{remark}[theorem]{Remark}
\newtheorem{example}[theorem]{Example}

\title[Mock Seifert matrices and unoriented algebraic concordance]{Mock Seifert matrices and \\ unoriented algebraic concordance}
\author[Hans U. Boden]{Hans U. Boden}
\address{Mathematics \& Statistics, McMaster University, Hamilton, Ontario}
\email{boden@mcmaster.ca}

\author[Homayun Karimi]{Homayun Karimi}
\address{Mathematics \& Statistics, McMaster University, Hamilton, Ontario}
\email{karimih@mcmaster.ca}


\begin{document}

\begin{abstract}

A mock Seifert matrix is an integral square matrix representing the Gordon-Litherland form of a pair $(K,F)$,
where $K$ is a knot in a thickened surface and $F$ is an unoriented spanning surface for $K$.
Using these matrices, we introduce a new notion of \textit{unoriented algebraic concordance}, as well as a new group 
 denoted $\mG^\ZZ$ and called the unoriented algebraic concordance group. This group is abelian and infinitely generated. There is a surjection $\lambda \colon \vC \to \mG^\ZZ$, where $\vC$ denotes the virtual knot concordance group. 
Mock Seifert matrices can also be used to define new invariants, such as the mock Alexander polynomial and mock Levine-Tristram signatures. These invariants are applied to questions about virtual knot concordance, crosscap numbers, and Seifert genus for knots in thickened surfaces. For example, we show that $\mG^\ZZ$ contains a copy of $\ZZ^\infty \oplus (\ZZ/2)^\infty \oplus(\ZZ/4)^\infty.$ 
\end{abstract}

\maketitle
 
\section*{Introduction} \label{section-1}

Let $\Si$ be a compact, connected, closed, oriented surface and $I=[0,1]$.
Let $K\subset \Si \times I$ be a $\ZZ/2$ null-homologous knot and $F \subset \Si \times I$ a spanning surface for $K$.
The surface $F$ may or may not be orientable, but we will regard it as \textit{unoriented}. Associated to the pair $(K,F)$ is a bilinear form $\cL_F \co H_1(F) \times H_1(F) \lto \ZZ$ called the Gordon-Litherland form. A \textit{mock Seifert matrix} is any square integral matrix representing  $\cL_F$.  (The term ``mock'' is meant to distinguish such matrices from the ones obtained in the usual way from oriented spanning surfaces.) We use these matrices to introduce a new notion of algebraic concordance and to construct a new group, denoted $\mG^\ZZ$ and called the \textit{unoriented algebraic concordance group}. 
The group $\mG^\ZZ$ is a concordance group of \textit{admissible} mock Seifert matrices, namely those arising from spanning surfaces $F$ with Euler number $e(F)=0.$
The group captures subtle concordance information carried by unoriented spanning surfaces, including information not present in the classical theory of knot concordance.  

We also construct a surjective homomorphism $\la \co \vC \lto \mG^\ZZ$, where $\vC$ denotes the concordance group of virtual knots.  This is the analogue of the Levine homomorphism in the virtual setting. The mock Seifert matrices are used to define new invariants of $\ZZ/2$ null-homologous knots in thickened surfaces, including the mock Alexander polynomial $\De_{K,F}(t)$ and mock Levine-Tristram signatures $\si_{K,F}(\om)$. The mock Alexander polynomial $\De_{K,F}(t)$ is well-defined up to multiplication by $t^k(t-1)^\ell$ for $k,\ell \in \ZZ$. Both $\De_{K,F}(t)$ and  $\si_{K,F}(\om)$ depend on the choice of spanning surface $F$, but only up to $S^*$-equivalence. 

These invariants are applied to questions about knot concordance, crosscap numbers, and the Seifert genus for knots in thickened surfaces. For instance, when $K$ is slice, we show that the mock Alexander polynomial satisfies a Fox-Milnor condition (\Cref{Fox-Milnor}), and that the mock Levine-Tristram signatures vanish (\Cref{LT-vanish}). We also observe that the group $\mG^\ZZ$ is abelian but not finitely generated; in fact we show that it contains a copy of $\ZZ^\infty \oplus (\ZZ/2)^\infty \oplus(\ZZ/4)^\infty$ (Propositions \ref{prop-inf-lin-ind} and \ref{prop-Z2-Z4}).

The mock Seifert matrices provide presentation matrices for the first homology $H_1(X_2)$ of the double branched cover of $\Si \times I$ branched along $K$ (\Cref{thm:double-branched}).  This result holds more generally
for $\ZZ/2$ null-homologous links in thickened surfaces. Consequently, the determinant of the mock Seifert matrix is independent of the choice of spanning surface and is, in fact, a welded invariant of links.

One advantage of working with mock Seifert matrices is that the resulting invariants can be computed using \textit{unoriented} spanning surfaces. This has practical value, especially since many of the knots under consideration do not admit orientable spanning surfaces. For classical knots, the theory degenerates in that the mock Seifert matrices become symmetric and carry only limited information. For example, the mock Alexander polynomials collapse to a single numerical invariant (the knot determinant), and the mock Levine-Tristram signatures specialize to a single value (the knot signature). So for knots in thickened surfaces of genus $g>0$, it is surprising that the mock Seifert matrices provide such powerful invariants, especially given that those same invariants degenerate when $g=0$, i.e., for classical knots.


Here is a brief outline of the contents of the rest of this paper.
In \Cref{section-1}, we review basic notions for knots in thickened surfaces, virtual knots, and spanning surfaces.
In \Cref{section-2}, we define concordance for knots in thickened surface and recall the construction of the concordance group of virtual knots. 
In \Cref{section-3}, we introduce the Gordon-Litherland form, mock Seifert matrices, and give necessary and sufficient conditions on a matrix to be a mock Seifert matrix for a knot. We also show that for links $L \subset \Si \times I$, the mock Seifert matrix is a presentation matrix for the first homology group of the double cover $X_2$ branched along $L$. 
In \Cref{section-4}, we introduce knot invariants derived from mock Seifert matrices, including the mock Alexander polynomial and mock Levine-Tristram signatures. 
In \Cref{section-5}, we construct the unoriented algebraic concordance group $\mG^\ZZ$ and define a
surjective homomorphism $\la \co \vC \lto \mG^\ZZ$. 
We prove that $\mG^\ZZ$ contains a subgroup isomorphic to $\ZZ^\infty \oplus (\ZZ/2)^\infty \oplus(\ZZ/4)^\infty.$
In \Cref{section-7}, we apply parity projection to describe certain natural subgroups of $\vC$ such as $\vC_2$, the group consisting of $\ZZ/2$ homologically trivial virtual knots. Stable parity projection induces a surjection $\varphi_2\co \vC \lto \vC_2.$

The commutative diagram in \eqref{cd-full} below summarizes the relationships between various concordance groups studied in this paper. In \eqref{cd-full}, $\cC$ is the classical concordance group and $\cG^\ZZ$ is the classical algebraic concordance group.
We conclude this introduction with a brief discussion of some known results and open problems.  

\begin{equation} \label{cd-full}
\begin{diagram}
 \cC &&&\rTo^{\text{\tiny Levine}} &&&\cG^\ZZ\\
\dInto  &&&&&&         \dTo\\
\vC         &\rOnto^{\text{\tiny $\varphi_2$}}   &\vC_2 &&\rOnto^{\text{\tiny $\lambda$}}&& \mG^\ZZ
\end{diagram}
 \end{equation}

\medskip



To begin, it is well-known that $\cG^\ZZ \cong \ZZ^\infty \oplus (\ZZ/2)^\infty \oplus(\ZZ/4)^\infty$ \cite{Levine-1969-b, Stoltzfus},
but the algebraic structure of  $\cC$ and $\vC$ remains mysterious.
A recent result of Chrisman implies that $\vC$ is non-abelian \cite{Chrisman-2020}.
The first vertical map in \eqref{cd-full} is induced by inclusion and is injective \cite{Boden-Nagel-2016}.
Its image lies in the center of $\vC$.
The second vertical map in \eqref{cd-full} is given by $A \mapsto A + A^\tr$; it is neither injective nor surjective.

It is an open question whether $\cC$ or $\vC$ contains torsion apart from the elements of order two represented by amphicheiral classical knots. Also open is whether $\cC$ or $\vC$ contains infinitely divisible elements. Might $\cC$ or $\vC$ contain a copy of $\QQ$ or $\QQ/\ZZ$? 

The concordance group of classical knots $\cC$ maps injectively into the subgroup $\vC_0 \subset \vC $ of homologically trivial virtual knots. 
Stable parity projection induces a surjection $\varphi_0\co \vC \lto \vC_0$ (see \Cref{section-5}).
It is not known whether $\vC_0$ is abelian.

In \cite{Chrisman-Mukherjee}, Chrisman and Mukherjee study the algebraic concordance order of almost classical knots.  They introduce an algebraic concordance group of \textit{Seifert pairs}, denoted $(\vG,\vG)^\ZZ$ (cf. \cite[Definition 2.5.4]{Chrisman-Mukherjee}). 
The construction in \Cref{S5-4} for the map $\la\co \vC_2 \lto \mG^\ZZ$ can be used to define a Levine-type surjection $\vC_0 \lto (\vG,\vG)^\ZZ$. There is also a natural map $(\vG,\vG)^\ZZ \lto \mG^\ZZ$ defined on Seifert pairs by sending $(V^+,V^-)$ to the mock Seifert matrix $A=V^+ +V^-$.


\subsection*{Notation} Unless otherwise specified, all homology groups are taken with $\ZZ$ coefficients.
Spanning surfaces are assumed to be compact and connected but not necessarily orientable. For a compact surface $F$, we use $b_1(F)$ to denote the rank of $H_1(F)$. Decimal numbers refer to virtual knots in Green's tabulation \cite{Green}. 

\section{Basic notions} \label{section-1}
In this section, we review the basic notions for knots in thickened surfaces and virtual knots, including spanning surfaces, $S^*$-equivalence, long virtual knots, and the operation of connected sum.

\subsection{Knots in thickened surfaces} \label{S1-1}
Let $\Si$ be a compact, connected, oriented surface and $I = [0,1]$ the unit interval. A \textit{\textbf{knot}} in $\Si \times I$  is an embedding of $S^1$ into the interior of $\Si \times I$, considered up to orientation-preserving homeomorphisms of the pair $(\Si \times I, \Si \times \{0\}).$ 	

It is convenient to represent knots in $\Si \times I$ using knot diagrams on $\Si$. A \textit{\textbf{knot diagram}} is a regular immersion of $S^1$ in $\Si$ with finitely many double points or crossings, and two diagrams represent the same knot if they can be related by a finite sequence of Reidemeister moves. 

In knot diagrams, we draw the crossings with a solid line for the over-crossing arc and a broken line for the under-crossing arc. Each crossing is either positive or negative according to a comparison of the orientation at the crossing with that of the surface. Our convention is to orient the crossing with a basis of direction vectors, where the first vector is the directed overcrossing arc and the second is the directed undercrossing arc. For example, in the usual orientation on the plane, 
$\begin{tikzpicture}
\draw[line width=1.50pt] (.22,-.18) -- (.05,-.05);
\draw[line width=1.50pt,->](-.05,.05) -- (-.22,.18);
\draw[line width=1.50pt,->](-.22,-.18) -- (.22,.18);
\end{tikzpicture}$
would be positive and $\begin{tikzpicture}
\draw[line width=1.50pt,->] (.22,-.18) -- (-.22,.18);
\draw[line width=1.50pt,->](.05,.05) -- (.22,.18);
\draw[line width=1.50pt](-.22,-.18) -- (-.05,-.05);
\end{tikzpicture}$ would be negative.
The notion of positive and negative crossing is independent of the choice of orientation on the knot $K$, but they are switched under a change in orientation of $\Si.$

Let $K$ be an oriented knot in $\Si \times I$. The reverse, $K^r$, is the same knot with the opposite orientation. The mirror image, $K^m$, is the same knot viewed in $-\Si \times I$, where $-\Si$ indicates $\Si$ with its orientation changed. We will use $-K$ to denote the knot $K^{rm}=K^{mr}.$ 

A \textit{\textbf{link}} in $\Si \times I$ is an embedding of the disjoint union $S^1 \cup \cdots \cup S^1$ in the interior of $\Si \times I$, up to orientation-preserving homeomorphisms of the pair $(\Si \times I, \Si \times \{0\}).$ Given a link $L \subset \Si \times I$, let $X_L = (\Si \times I \sm L)/\Si \times \{1\}$ and define the \textit{\textbf{link group}}, denoted $G_L$, to be the fundamental group $\pi_1(X_L)$.
\subsection{Virtual knots} \label{S1-2}
A \textit{\textbf{virtual knot}} is an equivalence class of knots in thickened surfaces up to stable equivalence.  Let $p \co \Si \times I \to \Si$ be projection. Stabilization is the operation of adding a 1-handle to $\Si$, disjoint from $p(K)$, and destabilization is the opposite procedure. Two knots $K_0 \subset \Si_0 \times I$ and $K_1 \subset \Si_1 \times I$ are said to be \textit{\textbf{stably equivalent}} if one is obtained from the other by a finite sequence of stabilizations, destablizations, and orientation-preserving diffeomorphisms of the pairs $(\Si_0 \times I, \Si_0 \times \{0\})$ and $(\Si_1 \times I, \Si_1 \times \{0\}).$ 

It is convenient to represent virtual knots using virtual knot diagrams. A \textit{\textbf{virtual knot diagram}} is a regular immersion of $S^1$ in $\RR^2$ with two types of double points,  classical crossings
$\begin{tikzpicture}
\draw[line width=1.50pt] (.2,-.18) -- (-.2,.18);
\draw[line width=1.50pt](.2,.18) -- (.05,.05);
\draw[line width=1.50pt](-.2,-.18) -- (-.05,-.05);
\end{tikzpicture}$
and virtual crossings 
$\begin{tikzpicture}
\draw[line width=1.50pt] (.2,-.18) -- (-.2,.18);
\draw[line width=1.50pt](.2,.18) -- (-.2,-.18);
\draw[line width=0.75pt] (0,0) circle (3pt); 
\end{tikzpicture}$. 
Two virtual knot diagrams are equivalent if they can be related by a finite sequence of Reidemeister moves and detour moves. If $K$ is a knot in $\Si \times I$, then the image of $K$ in $\RR^2$ under an orientation-preserving immersion will determine  a virtual knot diagram corresponding to $K$.  Conversely, if $D$ is a virtual knot diagram, there is a canonical construction of a surface $\Si$ and a knot in $\Si \times I$ representing the same virtual knot \cite{Kamada-Kamada-2000}.  In \cite{Carter-Kamada-Saito}, Carter, Kamada, and Saito establish a one-to-one correspondence between virtual knots and stable equivalence classes of knots in thickened surfaces.

A \textit{\textbf{virtual link}} is a stable equivalence class of links in thickened surfaces. Any virtual link can be represented by a \textit{\textbf{virtual link diagram}}. Given a virtual link diagram, the Wirtinger presentation describes the link group $G_L$ in terms of generators and relators, see \cite[Section 2]{Boden-Gaudreau-Harper-2017}.

Suppose $J$ and $K$ are two oriented virtual knots (so $J \sqcup K$ is a two component virtual link).  The \textit{\textbf{virtual linking number}} $\vlk(J, K)$ is defined to be the algebraic count of the crossings where $J$ crosses over $K$. Note that virtual crossings do not contribute to $\vlk(J, K)$, and it is not generally symmetric.

\begin{figure}[!ht]
\centering\includegraphics[height=30mm]{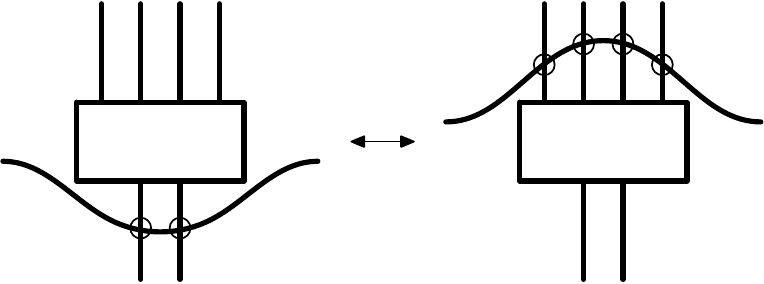}
\caption{The detour move.} \label{Fig:detour}
\end{figure}

\subsection{Spanning surfaces} \label{S1-3}
A \textit{\textbf{spanning surface}} for a knot $K \subset \Si \times I$ is a compact unoriented surface $F \subset \Si \times I$ with boundary $\partial F =K$. Not all knots in $\Si \times I$ admit spanning surfaces. In fact, a knot $K \subset \Si \times I$ admits a spanning surface if and only if it is  $\ZZ/2$ null-homologous, namely if $[K]=0$ in $H_1(\Si\times I; \ZZ/2)$. A \textit{\textbf{Seifert surface}} for a knot $K$ is an oriented spanning surface for it. A knot $K \subset \Si \times I$ admits a Seifert surface if and only if it is  $\ZZ$ null-homologous, namely if $[K]=0$ in $H_1(\Si\times I; \ZZ).$

A knot $K\subset \Si \times I$ is said to be \textit{\textbf{checkerboard colorable}} if it admits a diagram $D$ in $\Si$ such that the regions of $\Si \sm D$ can be colored black and white so that adjacent regions have different colors. Given a diagram with checkerboard coloring, one can construct a spanning surface by attaching half-twisted bands to the black (or white) regions. Clearly, any knot $K\subset \Si \times I$ that admits a spanning surface is checkerboard colorable. Thus, a knot is $\ZZ/2$ null-homologous if and only if it is checkerboard colorable.

Two spanning surfaces are said to be \textit{\textbf{S$^*$-equivalent}} if one can be obtained from the other by (i) ambient isotopy, (ii) attachment (or removal) of a tube, and (iii) attachment (or removal) of a half-twisted band.

When $g(\Si)>0$, the black and white checkerboard surfaces are not $S^*$-equivalent. Every spanning surface is $S^*$-equivalent to one of checkerboard surfaces (see \cite[Proposition 1.6]{Boden-Chrisman-Karimi-2021}). In particular, this implies that there are exactly two $S^*$-equivalence classes of spanning surfaces for every $\ZZ/2$ null-homologous knot in $\Si \times I$.

A virtual knot is said to be \textit{\textbf{checkerboard colorable}} if it can be represented by a $\ZZ/2$ null-homologous knot in a thickened surface, and it is said to be \textit{\textbf{almost classical}} if it can be represented by a $\ZZ$ null-homologous knot in a thickened surface.

It is easy to determine if a given virtual knot diagram $D$ is checkerboard colorable or almost classical by checking the \textit{indices} of the crossings of $D$. The \textit{\textbf{index}} of a crossing $c$ of $D$ is defined by setting $\ind(c) = \vlk(D',D'') - \vlk(D'',D')$, where $D'$ and $D''$ are the oriented virtual knot diagrams obtained from the oriented smoothing at $c$, and $D'$ contains the outgoing overcrossing arc of $c$. Note that $\ind(c)\in \ZZ$. The virtual knot diagram $D$ is checkerboard colorable if and only if $\ind(c)$ is even for all crossings, and it is almost classical if and only if $\ind(c)=0$ for all crossings.

\subsection{Long knots in thickened surfaces} \label{S1-4}
A \textit{\textbf{long knot}} in a thickened surface is an oriented knot $K$ in $\Si \times I$ together with a distinguished basepoint $q \in \Si$ such that $K$ passes through $q \times \{ \frac{1}{2} \}.$ We further assume that $q \in D$ for some 2-disk neighborhood $D=\{(x,y) \mid x^2+y^2 \leq 1\}$ contained in $\Si$ such that $K \cap (D \times I) = \{(x,0) \mid x \in [-1,1] \} \times \{\frac{1}{2} \}.$ Long knots are considered up to orientation-preserving homeomorphisms of the pair $(\Si \times I,\Si \times \{0\})$ that are the identity map on $D \times I$. We will use the term \textit{\textbf{round knot}} to refer to usual knots. Thus, there is a well-defined map from long knots to round knots given by simply forgetting the basepoint.

One advantage to working with long knots, at least for $\ZZ/2$ null-homologous knots, is that there is a way of using the basepoint to associate a  \textit{preferred} spanning surface. This is defined as follows.
\begin{definition} \label{defn:pref}
Let $(K,q)$ be a $\ZZ/2$ null-homologous long knot in $\Si \times I$. Then a spanning surface $F$ for $K$ is said to be \textit{\textbf{preferred}} if $F \subset (\Si \sm \{x\}) \times I$, where $x \in \Si$ is a marked point chosen to the right of the basepoint $q$ with respect to $K$. 
\end{definition}

Thus, any preferred spanning surface $F$ is necessarily disjoint from $\{x\} \times I$, where $x \in \Si$ denotes the marked point. The placement of the marked point is akin to the location of $\infty$ in $S^2.$

Note that any two preferred spanning surfaces for a long knot $(K,q)$ are necessarily $S^*$-equivalent. Thus, for long knots, there is a unique $S^*$-equivalence class of preferred spanning surfaces. For instance, if $D$ is a checkerboard colorable diagram with basepoint $q \in D,$ then the preferred checkerboard surface is the one that appears to the left of $q$. 

\subsection{Long virtual knots} \label{S1-5}
Long virtual knots can be defined as stable equivalence classes of long knots in thickened surfaces as in \Cref{S1-2}. There is a useful alternative description in terms of virtual knot diagrams which we give now. 

A \textit{\textbf{long virtual knot diagram}} is a regular immersion of $\RR$ in the plane $\RR^2$ which coincides with the $x$-axis outside some ball of large radius. It has finitely many double points, and each crossing is either classical $\begin{tikzpicture}
\draw[line width=1.50pt] (.2,-.18) -- (-.2,.18);
\draw[line width=1.50pt](.2,.18) -- (.05,.05);
\draw[line width=1.50pt](-.2,-.18) -- (-.05,-.05);
\end{tikzpicture}$
or virtual
$\begin{tikzpicture}
\draw[line width=1.50pt] (.2,-.18) -- (-.2,.18);
\draw[line width=1.50pt](.2,.18) -- (-.2,-.18);
\draw[line width=0.75pt] (0,0) circle (3pt); 
\end{tikzpicture}$. 
Two long virtual knot diagrams are equivalent if they can be related by a finite sequence of compactly supported planar isotopies, Reidemeister moves, and detour moves. 

Long virtual knot diagrams are oriented by convention from left to right. The long knot represented by the $x$-axis is called the long unknot.

Given a long virtual knot diagram $D$, its closure is the round knot diagram $\wh{D}$ obtained by joining two points of $D$ on the $x$-axis by a large semicircle that misses the rest of $D$. Closure gives a well-defined map from long virtual knots to round virtual knots. In \cite{Silver-Williams-2006c}, Silver and Williams show that any round virtual knot is the closure of infinitely many distinct long virtual knots. 

\subsection{Connected sum of virtual knots} \label{S1-6}
In this section, we recall the operation of connected sum for virtual knots and for knots in thickened surfaces. Given two virtual knot diagrams, one can construct the connected sum, but it depends on the points where the two diagrams are connected. Therefore, connected sum does not lead to a well-defined operation on virtual knots; it depends on the diagrams used as well as the placement of connection points. 

For long virtual knots, connected sum is well-defined. Given two long virtual knot diagrams $D_0$ and $D_1$, the connected sum is denoted $D_0 \# D_1$ and it is the long virtual knot diagram  obtained by concatenating the two knots with $D_0$ on the left and $D_1$ on the right. The operation respects Reidemeister equivalence and leads to a well-defined operation on long virtual knots. The set of long virtual knots forms a monoid under connected sum, and the identity element is the long unknot. This monoid is not commutative, see \cite[Theorem 9]{Manturov-2008}.

For knots in thickened surfaces, the situation is similar. Given two knot diagrams $D_0$ on $\Si_0$ and $D_1$ on $\Si_1$, 
the connected sum is defined but depends on how the diagrams are connected. This does not lead to a well-defined operation; the result depends on the diagrams used as well as the placement of the connection points. 

For long knots in thickened surfaces, connected sum is well-defined.  Given two long knots $(K_0,q_0)\subset \Si_0\times I$ and $(K_1,q_1)\subset \Si_1\times I$ in thickened surfaces, the connected sum is the long knot in $(\Si_0 \# \Si_1)\times I$ constructed by removing small disks centered at $q_i$ from $\Si_i$ for $i=0,1$, attaching a cylinder $S^1 \times I$ to obtain  $\Si_0 \# \Si_1$, and letting $K_0 \# K_1$ be the knot obtained by connecting $K_0$ to $K_1$ using the two arcs $\{\pm 1\} \times I$. The orientations of $K_0$ and $K_1$ are preserved in this construction, and the basepoint of the long knot $(K_0 \# K_1, q)$ is chosen to be the point $(1,1/2)$ on the cylinder $S^1 \times I$.   

\section{The concordance group of virtual knots} \label{section-2}
In this section, we introduce two notions of concordance, one for virtual knots and another for knots in thickened surfaces. Under connected sum, the concordance classes of long virtual knots form a group denoted $\vC$ and called the  concordance group of virtual knots.

The main results established are Theorems \ref{thm:longknotconc} and \ref{slice-concordance}. The first shows that two long virtual knots are concordant if and only if their connected sum is slice. The second gives an analogous result for $\ZZ/2$ null-homologous knots in thickened surfaces.

We begin with a diagrammatic definition of concordance for virtual knots due to Kauffman \cite{Kauffman-2015}. This definition applies to both round and long virtual knots.  

\begin{definition} \label{defn:virtconc} 
Two virtual knot diagrams $D_0$ and $D_1$ are said to be \textit{\textbf{virtually concordant}} if $D_0$ can be transformed into $D_1$ by a finite sequence of $b$ births, $d$ deaths, $s$ saddle moves,  Reidemeister moves, and detour moves, such that $s=b+d$. 

A round or long virtual knot is \textit{\textbf{virtually slice}} if it is virtually concordant to the unknot.
\end{definition} 

\begin{remark}
\begin{enumerate}
\item[(i)] If two long virtual knots are virtually concordant, then their closures are virtually concordant as round virtual knots.
\item[(ii)] The converse to (i) is not true. There exist long virtual knots which are not virtually concordant but whose closures are virtually concordant as round virtual knots. 
\item[(iii)] A long virtual knot is virtually slice if and only if its closure is virtually slice as a round virtual knot, see \cite[Lemma 3.3]{Boden-Nagel-2016}.
\end{enumerate}
\end{remark}

The operation of connected sum is associative on long virtual knots and induces a well-defined group operation on concordance classes.  For a long virtual knot  $K$, let $-K$ be the knot obtained by changing the crossings of $K$ and reversing the orientation. Then $K \#(-K)$ can be seen to be virtually slice (see \cite[Theorem 1.2]{Chrisman-2016}), and so $-K$ is an inverse for $K$ up to concordance.  Thus, the concordance classes of long virtual knots form a group under connected sum which is denoted $\vC$ and called the \textit{\textbf{concordance group of virtual knots}}.

\begin{theorem} \label{thm:longknotconc}
Let $J$ and $K$ be long virtual knots. Then  $J$ and $K$ are virtually concordant if and only if $J\# (-K)$ is virtually slice.
\end{theorem} 

\begin{proof}
Let $\simeq$ denote virtual concordance. Given long virtual knots $J, J', K,K'$, with
$J\simeq J'$ and $K\simeq K'$, then it follows that $J\#K \simeq J'\# K'.$
Thus, if $J \simeq K$, then $J \# (-K) \simeq J \# (-J)$, which is virtually slice.
This gives one direction.

Now suppose $J\# (-K)$ is virtually slice. Since $(-K) \# K$ is also virtually slice, we have
$$J \simeq J \# ((-K) \# K) =(J \# (-K)) \# K \simeq K.$$ This gives the other direction and completes the proof.
\end{proof}

Next, we recall a definition of concordance for knots in thickened surfaces due to Turaev (cf. \cite[Section 2.1]{Turaev-2008-A}). 

\begin{definition} \label{defn:conc} 
Two oriented knots $K_0 \subset \Si_0 \times I$ and $K_1 \subset \Si_1\times I$ are said to be \textit{\textbf{virtually concordant}} if there exists a compact, oriented 3-manifold $W$ with $\partial W =-\Si_0 \cup \Si_1$ and an annulus $C$  properly embedded in $W \times I$ such that $\partial C = -K_0 \sqcup K_1$. 

A knot $K \subset \Si \times I$ is said to be \textit{\textbf{virtually slice}} if there exists a compact, oriented 3-manifold $W$ with $\partial W =\Si$ and a disk $D$ properly embedded in $W \times I$ with $\partial D =K.$ 
\end{definition}

We also introduce a notion of concordance for spanning surfaces of knots in thickened surfaces (cf. \cite[Definition 5.1]{Boden-Karimi-2021}).

\begin{definition} \label{defn-spanning-surf}
Let $K_0 \subset \Si_0\times I$ and $K_1 \subset \Si_1\times I$ be knots with spanning surface $F_0$ and $F_1$, respectively. Then $F_0$ and $F_1$ are said to be \textit{\textbf{concordant}} if there exists a compact oriented 3-manifold $W$ with $\partial W = -\Si_0 \cup \Si_1$ and a properly embedded annulus $C \subset W \times I$ with boundary $\partial C = -K_0 \cup K_1$ such that the closed surface $F_0 \cup C \cup F_1$ bounds a compact unoriented 3-manifold $V$ embedded in $W \times I$.
\end{definition}

The next theorem gives equivalent conditions for $\ZZ/2$ null-homologous long knots to be concordant.
The first is in terms of sliceness of their connected sum, and the second is in terms of admitting preferred spanning surfaces that are concordant in the sense of \Cref{defn-spanning-surf}.

\begin{theorem}\label{slice-concordance}
Let $(K_0,q_0)\subset \Si_0\times I$ and $(K_1,q_1)\subset \Si_1\times I$ be $\ZZ/2$ null-homologous long knots. Then  the following are equivalent:
\begin{enumerate}
\item[(i)] $(K_0,q_0)$ and $(K_1,q_1)$ are virtually concordant as long knots;
\item[(ii)] $-K_0\# K_1$ is a virtually slice knot;
\item[(iii)] $(K_0,q_0)$ and $(K_1,q_1)$ admit preferred spanning surfaces $F_0\subset \Si_0\times I$  and $F_1\subset \Si_1\times I$, respectively, which are concordant as spanning surfaces. 
\end{enumerate}  
\end{theorem} 

\begin{proof}
The equivalence follows by showing that (i) $\Rightarrow$ (ii) $\Rightarrow$ (iii) $\Rightarrow$ (i). 
Note that (i) $\Leftrightarrow$ (ii) is the analogue of \Cref{thm:longknotconc} for knots in thickened surfaces.

Claim: (i) $\Rightarrow$ (ii). Suppose that $(K_0,q_0)$ and $(K_1,q_1)$ are virtually concordant long knots. Then we have an oriented 3-manifold $W$ with boundary  $\partial W = -\Si_0\cup \Si_1$, and a properly embedded annulus $C\subset W\times I$ such that $C \cap (\Si_i \times I) = K_i$ for $i=0,1.$ Then $W$ contains a 3-ball of the form $B^3=D \times [0,1]$ such that the intersection $B^3 \cap \Si_i =D \times \{i\}$ is a 2-disk neighborhood containing $q_i$ for $i=0,1.$ (Here, $D=\{(x,y) \mid x^2 +y^2\leq 1 \}$ is the standard 2-disk.) Let $W'$ be the 3-manifold obtained by removing $B^3$ from $W$. Then $\partial W'=\Si_0\# \Si_1$, and the intersection $C \cap (B^3 \times I)$ is the band $\{(x,0) \mid x \in [-1,1] \} \times [0,1] \times \{\frac{1}{2}\}$.  Removing this band from $C$ gives a disk in $W'\times I$ with boundary $-K_0\# K_1$, which shows that $-K_0\# K_1$ is virtually slice.

Claim: (ii) $\Rightarrow$ (iii). Suppose $-K_0\# K_1$ is virtually slice. Let $F \subset (-\Si_0 \# \Si_1) \times I$ be a preferred spanning surface for $-K_0\# K_1$ with $e(F)=0$. Then there is an oriented 3-manifold $W$ with boundary  $-\Si_0\# \Si_1$, and a slice disk $D\subset W\times I$ with boundary $-K_0\# K_1$. By Theorem 2.3 in \cite{Boden-Karimi-2021}, we can choose $W$ and $D$ so that the closed surface $E=F\cup D$ bounds a 3-manifold $V$ embedded in $W\times I.$

Choose an annulus $S^1\times I$ in $-\Si_0\#\Si_1$ on which we have formed the connected sum. Attach a 3-dimensional 2-handle to $W$ whose attaching region is the given annulus. Let $W'$ denote the resulting oriented 3-manifold. Then it has boundary $\partial W' = -\Si_0\cup \Si_1$. 

On the other hand, in $W'\times I$, consider the thickened 2-handle $D^2 \times I \times I$. It contains a band $b$ which when  attached to the slice disk $D$ results in an annulus $C$ as seen in Figure \ref{fig-slice-concordance}. Under attachment of the band $b$, the connected sum $-K_0\# K_1$ changes to the disjoint union $K_0 \cup K_1$. It follows that $K_0$ and $K_1$ are virtually concordant. Now $F=F_0\#_{b} F_1$, where $F_0,F_1$ are spanning surfaces for $K_0,K_1$, respectively. It is easy to check that $F_0$ and $F_1$ are preferred spanning surfaces for $K_0$ and $K_1$, respectively. We have $F_0\cup C \cup F_1=F\cup D$. Since $F\cup D$ bounds the compact, unoriented 3-manifold $V \subset W \times I,$ this shows that  $F_0$ and $F_1$ are concordant as spanning surfaces. 

\begin{figure}[htbp]
\begin{center}
\includegraphics[scale=0.45]{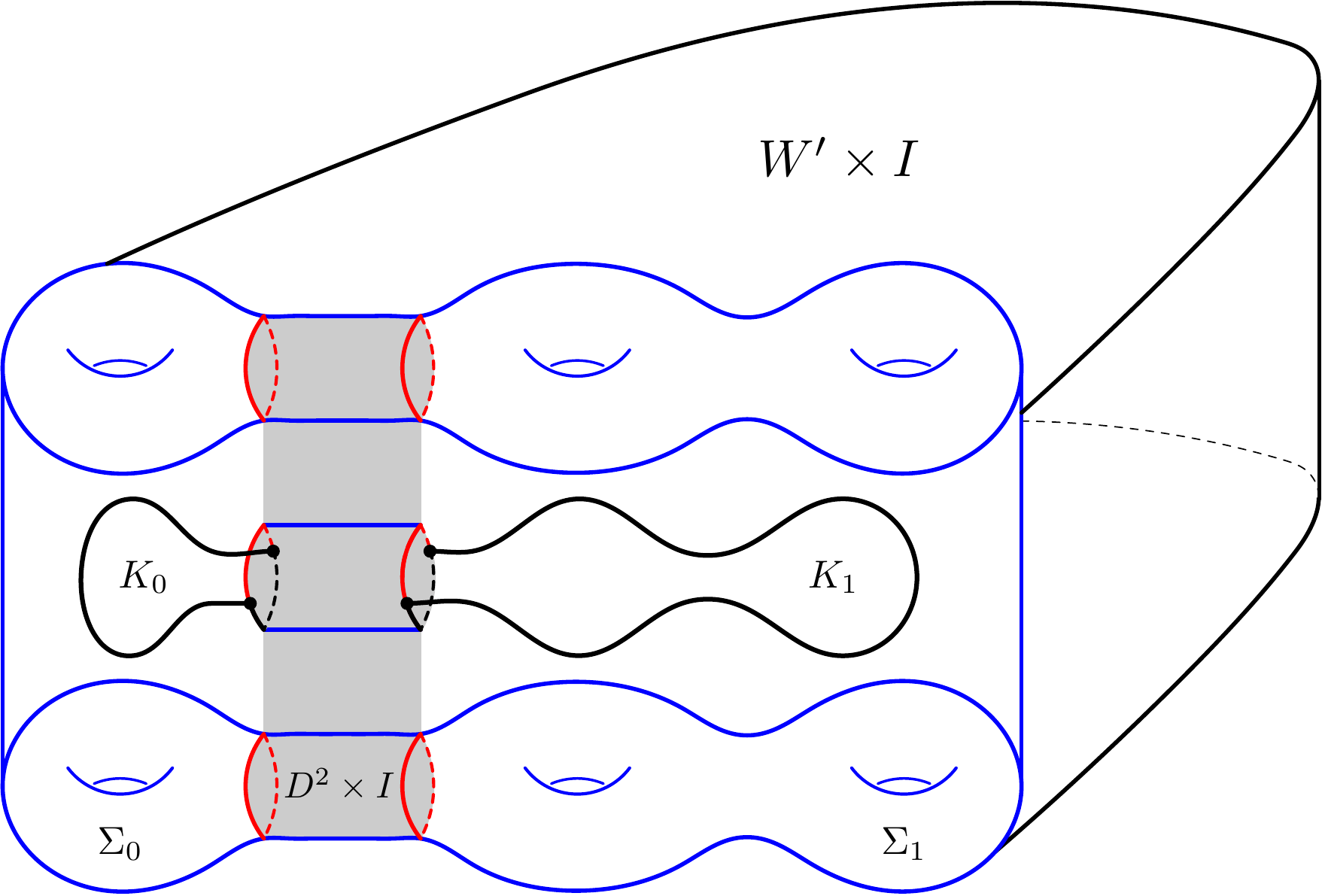}\hspace{.5cm}
\includegraphics[scale=0.45]{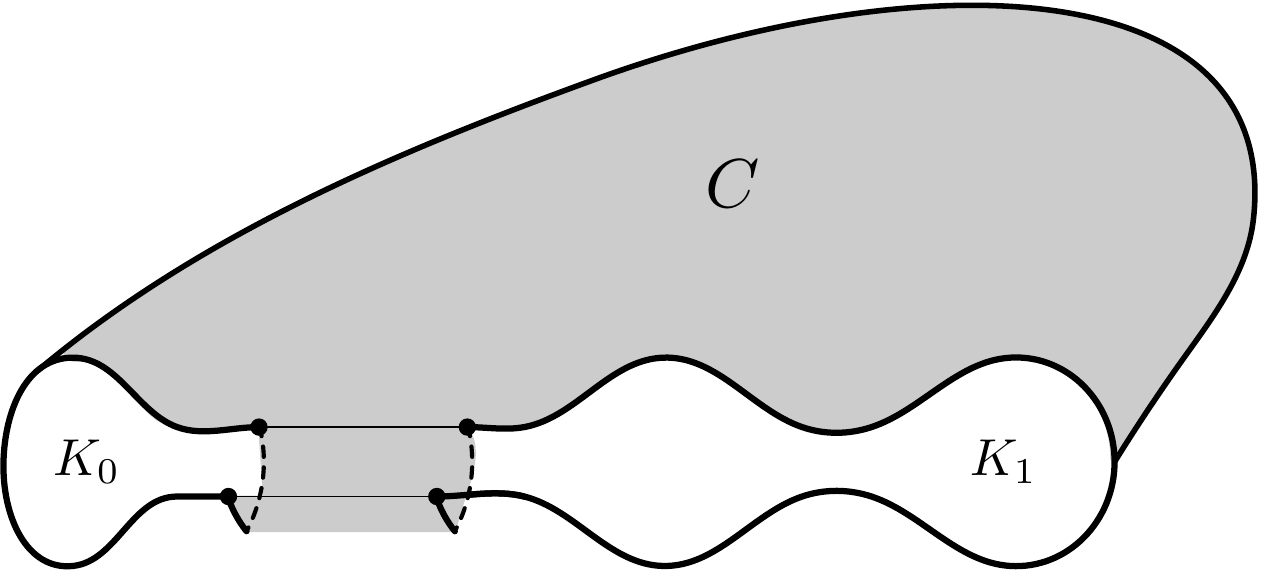}
\caption{\small On left, the manifold $W'\times I$. On right, a slice disk $D$ for $-K_0\# K_1$  with a band attached to form an annulus $C$.}
\label{fig-slice-concordance}
\end{center}
\end{figure}

Claim: (iii) $\Rightarrow$ (i). This follows directly from Definitions \ref{defn:conc} and \ref{defn-spanning-surf}.
\end{proof}

\section{Mock Seifert matrices} \label{section-3}
In this section, we define the Gordon-Litherland linking form and the associated mock Seifert matrix. We give a complete characterization of the set of mock Seifert matrices that occur for $\ZZ/2$ null-homologous knots in thickened surfaces. 

\subsection{The Gordon-Litherland form} \label{S3-1}
To begin, we recall the definition of the relative linking number, $\lk(J, K)$, for disjoint oriented simple closed curves $J, K$ in the interior of $\Si \times I$.

If $J \subset \Si \times I$ is a knot, then the relative homology group $H_1(\Si \times I \sm J, \Si \times \{1\})$ is infinite cyclic generated by a meridian $\mu$ of $J$ (for a proof, see \cite[Proposition 7.1]{Boden-Gaudreau-Harper-2017}).  Given a second knot $K \subset \Si \times I$ disjoint from $J$, let $[K]$
be its homology class in $H_1(\Si \times I \sm J, \Si \times\{1\})$. We define $\lk(J,K)$ to be the unique integer $m$ such that $[K] = m \mu.$
Equivalently, we can define $\lk(J,K) = J\cdot B$, where $B$ is a 2-chain in $\Si \times I$ such that $\partial B =  K - v$ for some  1-cycle $v$ in $\Si \times \{1\}$ and $\cdot$ denotes the intersection number. Relative linking is not symmetric, indeed by \cite[Section 1.2]{Cimasoni-Turaev}, we have 
$$\lk(J, K)-\lk(K, J) = p_*(J)\cdot p_*(K),$$
where $p_*\co H_1(\Si \times I)\to H_1(\Si)$ is the map induced by projection $p\co \Si \times I \to \Si$ and $\cdot$ denotes the intersection form on $\Si$. If $J$ and $K$ are given by diagrams on $\Si$, then $\lk(J, K)$ is simply the number of times $J$ crosses over $K$, counted with sign.

For a compact, connected, unoriented surface $F\subset \Si \times I$, its normal bundle $N(F)$ has boundary a $\{\pm 1\}$-bundle $\wt{F}\stackrel{\pi}{\lto}F$, a double cover with $\wt{F}$ oriented. The transfer map $\tau \co H_1(F) \to H_1(\wt{F})$ is defined by $\tau([\al]) = [\pi^{-1}(\al)].$ 

\begin{definition} \label{defn:GL-form}
The \textit{\textbf{Gordon-Litherland linking form}} is the map 
$$\cL_F \co H_1(F) \times H_1(F)  \lto \ZZ$$ 
given by $\cL_F(\al,\be) = \lk(\tau \al,\be)$  for $\al,\be \in H_1(F).$ 
\end{definition}

The  form in \Cref{{defn:GL-form}} is closely related to the Gordon-Litherland pairing, as defined in \cite{GL-1978} and  extended to links in thickened surfaces in \cite{Boden-Karimi-2023, Boden-Chrisman-Karimi-2021}. Given a compact, unoriented spanning suface $F \subset \Si \times I$, the Gordon-Litherland pairing is denoted
$$\cG_F\co H_1(F) \times H_1(F)  \lto \ZZ$$ 
and defined  by setting  
$$\cG_F(\al,\be) = \tfrac{1}{2}\left(\lk(\tau \al,\be) +\lk(\tau \be,\al)\right)$$ 
(see \cite[\S 2]{Boden-Karimi-2023}). The pairing $\cG_F$ is clearly symmetric, and in fact it is the symmetrization of the form $\cL_F$ from  \Cref{defn:GL-form}.

\begin{remark}
Although the form $\cL_F$ in \Cref{defn:GL-form} is not symmetric, its restriction to $\Ker(H_1(F) \to H_1(\Si \times I))$ is symmetric.
In fact,  $\cL_F(\al,\be)=\cL_F(\be,\al)$ whenever $\al \in \Ker(H_1(F) \to H_1(\Si \times I))$. This follows from  the formula $\cL_F(\al,\be)=\cG_F(\al,\be) + p_*(\al)\cdot p_*(\be)$, the fact that $\cG_F$ is symmetric, and the observation that $p_*(\al)\cdot p_*(\be)=0$ whenever $\al \in \Ker(H_1(F) \to H_1(\Si \times I))$. In particular, if $K \subset S^2 \times I$ is classical, then $\cL_F$ is symmetric.
\end{remark}

In \cite{Boden-Chrisman-Karimi-2021},  the pairing $\cG_F$ is used to define certain invariants (signature, determinant, and nullity) for links in thickened surfaces. The invariants are given by 
\begin{equation} \label{sig-det-null}
\begin{split}
\si(K,F)&=\sig(\cG_{F})+\tfrac{1}{2} e(F), \\
\det(K,F)&=|\det(\cG_{F})|, \\
n(K,F)&=\text{nullity}(\cG_{F}), 
\end{split}
\end{equation}
where $e(F)$ is the normal Euler number of $F$ given  by 
\begin{equation}\label{eqn-Euler}
e(F) = -\lk(K, K').
\end{equation}
Here $K'$ is a longitude for $K$ that misses $F$ and oriented compatibly with $K$. The Euler number $e(F)$ is even and does not depend on the orientation of $K$. 

The invariants $\si(K,F)$, $\det(K,F)$, $n(K,F)$ in \eqref{sig-det-null} depend on the spanning surface $F$, but only on its $S^*$-equivalence class. Since knots $K \subset \Si \times I$ admit two $S^*$-equivalence classes of spanning surfaces, they typically have two signatures, two determinants, and two nullities.  For long knots, one can associate a single signature, determinant, and nullity by working with the $S^*$-equivalence class of their preferred spanning surfaces. 

The behavior of the signature invariant under virtual concordance was studied in \cite{Boden-Karimi-2021}.
The next result is an immediate consequence.

\begin{proposition}
Suppose $(K_0,q_0)$ and $(K_1,q_1)$ are $\ZZ/2$ null-homologous long knots with $\det(K_0)\neq 0$ and $\det(K_1)\neq 0$. If $(K_0,q_0)$ and $(K_1,q_1)$ are virtually concordant as long knots, then $\si(K_0)=\si(K_1)$.
\end{proposition}

\begin{proof}
By \Cref{slice-concordance}, $-K_0\# K_1$ is slice. Since $n(-K_0\# K_1)=n(K_0)+n(K_1)=0$, then Theorem 3.2 in \cite{Boden-Karimi-2021} implies that $\si(-K_0\# K_1)=\si(-K_0)+\si(K_1)=0$. The results follows from \cite[Proposition 5.7]{Boden-Chrisman-Karimi-2021}.
\end{proof}

\begin{corollary}
If $K \subset \Si \times I$ is a $\ZZ/2$ null-homologous knot with spanning surfaces $F_0, F_1$ with different signatures, i.e., such that $\si(K,F_0) \neq \si(K, F_1)$, then $K$ is not virtually concordant to a classical knot.
\end{corollary}

The invariants in \eqref{sig-det-null} correspond to certain combinatorial invariants of checkerboard colorable virtual knots defined in terms of Goeritz matrices by Im, Lee, and Lee in \cite{Im-Lee-Lee-2010}. Under the correspondence, the black and white surfaces switch roles, and this is an instance of the principle of \textit{\textbf{chromatic duality}}; see Theorem 5.4, \cite{Boden-Chrisman-Karimi-2021}.

\subsection{Mock Seifert matrices} \label{S3-2}
In this section, we introduce mock Seifert matrices. We show that every mock Seifert matrix has odd determinant. The converse is true and will be proved in \Cref{S3-4}.
 
\begin{definition} \label{defn:mock-Seifert}
If $K\subset \Si \times I$ is a $\ZZ/2$ null-homologous knot with spanning surface $F \subset \Si \times I$, then an $n \times  n$ integral matrix is said to be a \textit{\textbf{mock Seifert matrix}} if it has $ij$ entry equal to $\lk(\tau \al_i,\al_j)$, where $\{\al_1,\ldots, \al_n\}$ is some basis for $H_1(F)$.  
\end{definition}
Since the spanning surface $F$ is not necessarily orientable, it may not be a Seifert surface.
Therefore, the matrix $A$ may not be a Seifert matrix. This explains why we refer to it as a
\textit{mock} Seifert matrix. 
It represents the linking form $\cL_F$ with respect to some basis for $H_1(F)$, and under a change of basis, the matrix changes by unimodular congruence, i.e., $A$ changes to $P^\tr AP$ for a unimodular matrix $P$.

Note that if $A$ is a mock Seifert matrix representing the form $\cL_F$ with respect to some basis for $H_1(F)$, then its symmetrization $(A+A^\tr)/2$ is an integral matrix representing the pairing $\cG_F$.   

Let $F$ be a compact surface with connected boundary and
$$H = \begin{bmatrix}0 & 1 \\ 1 & 0\end{bmatrix}
\quad \text{ and } \quad I_n =\begin{bmatrix}1 & & \textbf{0} \\ & \ddots \\ \textbf{0} & & 1\end{bmatrix},$$
the $ 2 \times 2$ hyperbolic matrix and the $ n \times n$ identity matrix, respectively. Then according to the classification of compact surfaces,  the mod 2 intersection pairing of $F$ is represented either by $H^{\oplus g}$ or $I_n$, depending on whether $F$ is orientable or not. Here, $H^{\oplus g}$ is shorthand notation for the block diagonal $2g \times 2g$ matrix  $\bigoplus_{i=1}^g H$.

In the following, all matrices are assumed to have integer entries. Let $A$ be an square integral matrix.
\begin{proposition}\label{GL-matrix}
If $A$ is the mock Seifert matrix for some $\ZZ/2$ null-homologous knot in a thickened surface, then its mod 2 reduction is the intersection matrix for some compact surface with connected boundary.
\end{proposition}

\begin{proof}
Suppose that $A$ occurs as the mock Seifert matrix for a knot $K \subset \Si \times I$ with spanning surface $F$. If $F$ is orientable, then it is a once-punctured surface of genus $g$. If $F$ is not orientable, then it is a once-punctured surface as the connected sum of $n$ projective planes. 

Let $\wbar{A}$ be the $\ZZ/2$ matrix obtained from the mod 2 reduction of $A$. We will show that, up to congruence, $\wbar{A}$ is equal to either $H^{\oplus g}$ or $I_n$, depending on whether $F$ is orientable or not.

To see this, we write $F$ as a 2-disk with $n$ bands attached. The bands may be twisted and knotted, and they may cross one another. If $F$ is orientable, the bands are attached with their feet alternating, and each band has an even number of half-twists. If $F$ is not orientable, the bands are attached with their feet unlinked and each band has an odd number of half-twists. For $i=1,\ldots, n$, let $\al_i$ be the simple closed curve obtained by connecting the endpoints of the core of the $i$-th band with a path in the disk. Then $\{\al_1,\ldots, \al_n\}$ is a basis for $H_1(F)$. Up to congruence, $A$ is equal to the matrix with $ij$ entry $\lk(\tau \al_i, \al_j).$

If the $i$-th band crosses over the $j$-th band, then it alters $A_{ij}$ by $\pm 2$. (This is because $\tau \al_i$ effectively goes over $\al_j$ twice.)  Thus, $\wbar{A}$ is not affected by band crossings.  Thus, the nonzero entries in $\wbar{A}$ come from odd twisting of the bands and intersection points of $\al_i$ and $\al_j$ in the disk. If $F$ is orientable, then $\wbar{A} = H^{\oplus g}$. If $F$ is non-orientable, then $\wbar{A} = I_n$. This completes the proof.
\end{proof}

Of course, a square integral matrix $A$ is the intersection matrix for a compact surface with connected boundary if and only if its mod 2 reduction $\wbar{A}$ is non-singular. Notice that $\wbar{A}$ is non-singular $\Leftrightarrow \det \wbar{A} = 1 \Leftrightarrow  \det A$ is odd. The next result is now a direct consequence.

\begin{corollary}\label{odd-det}
If $A$ is the mock Seifert matrix for some $\ZZ/2$ null-homologous knot in a thickened surface, then $\det A$ is odd.
\end{corollary}
 
\begin{remark}
If $F$ is orientable, then it determines a Seifert surface for $K$, and we can deduce this using an alternative argument.  Let $\{\al_1,\ldots,\al_n\}$ be a basis for $H_1(F)$ and $V^\pm =\lk(\al_i^\pm,\al_j)$ be the positive and negative linking matrices.  By Definition 7.6 of \cite{Boden-Gaudreau-Harper-2017}, the knot $K$ has Alexander polynomial $\De_K(t) = \det(tV^- - V^+)$. Clearly $A = V^+ + V^-.$ Therefore, $\det A= \De_K(-1)$, which is odd by Lemma 7.12 of \cite{Boden-Gaudreau-Harper-2017}. 
\end{remark}

\subsection{Double branched covers} \label{S3-3}
In this section, we show that the mock Seifert matrix of a $\ZZ/2$ null-homologous link $L \subset \Si \times I$ is a presentation matrix for $H_1(X_2)$, where  $X_2$ is the double cover branched along $L$ constructed below.  

We start by recalling the construction of the double cover $\wh{X}_2$ and the double branched cover $X_2$. Let $X = (\Si \times I \sm L) / \Si \times \{1\}$ be the space obtained from $\Si \times I \sm L$ by collapsing $\Si \times \{1\}$ to a point. Then $X$ is a finite CW complex with fundamental group  $\pi_1(X) \cong G_L$ and first homology $H_1(X) \cong H_1(\Si \times I \sm L, \Si \times \{1\}).$ Given a spanning surface $F$ for $L$, we can define a homomorphism $\phi \co \pi_1(X) \lto \ZZ/2$ by sending any simple closed curve $\ga$ in $X$ to its  mod 2 intersection number with $F$. Then $\phi$ is well-defined and surjective. Let $\wh{X}_2$ be the associated double cover of $X$ with $\pi_1(\wh{X}_2)$ isomorphic to $\ker(\phi)$. 

The double cover $\wh{X}_2$ can also be constructed geometrically as follows.
Let $F$ be a spanning surface for $L$ and $N(F)$ a closed regular neighborhood of $F.$ Then $N(F)$ is orientable and has boundary  $\partial N(F)=\wt{F}$, the oriented double cover of $F$ in case $F$ is non-orientable. Further, $N(F)$ is the mapping cylinder of the covering map $\pi \colon \wt{F} \to F$ and hence is an $I$-bundle over $F$. If $\ga$ is a closed loop in $F$, then $\pi^{-1}(\ga)$ is a single loop if $\ga$ is orientation reversing and a union of two loops if $\ga$ is orientation preserving.

Let $W = X \sm \Int N(F)$, the result of cutting $X$ along $F$.
If $F$ is non-orientable, then $\wt{F}$ is connected. Let $t\colon \wt{F}\to \wt{F}$ be the map that interchanges the two end points of each fibre of the above $I$-bundle. Then $t$ is a homeomorphism  with $t^2 =1$, and $X=W/\!\sim$ where $x \sim t(x)$ for $x \in \wt{F}= \partial N(F)$. 

The double cover $\wh{X}_2$ is constructed by taking two copies $W_0$ and $W_1$ of $W$, along with copies $\wt{F}_0 \subset W_0$ and $\wt{F}_1 \subset W_1$ of $\wt{F}$, and identifying $x \in \wt{F}_0$ with $tx$ in $\wt{F}_1$.  The double branched cover $X_2$ is obtained from $\wh{X}_2$ by attaching solid tori to its boundary components.  

\begin{theorem} \label{thm:double-branched}
Let $L \subset \Si \times I$ be a $\ZZ/2$ null-homologous link and let $X_2$ be the double cover of $\Si \times I$ branched along $L$. If $F \subset \Si \times I$ is a spanning surface for $L$ with Gordon-Litherland form $\cL_F \co H_1(F) \times H_1(F)  \to \ZZ$ and mock Seifert matrix $A$, then $A$ is a presentation matrix for $H_1(X_2)$.
\end{theorem}

\begin{proof} 
If $F'$ is obtained by adding a half-twisted band to $F$, then their mock Seifert matrices are related by $A'= A \oplus [\pm 1]$. In particular, the matrices $A$ and $A'$ present isomorphic modules. Therefore, if one of them is a presentation matrix for $H_1(X_2)$, then the other is too.  Thus, we can assume $F$ is non-orientable.

For any compact, connected surface $F$ in the interior of $\Si \times I$, the homology groups $H_1(\Si\times I \sm F,\Si\times \{1\})$ and $H_1(F)$ are isomorphic; both are free abelian of  the same rank. In fact, there is a unique non-singular bilinear form
$$\varphi \co H_1(\Si\times I \sm F,\Si\times \{1\}) \times H_1(F ) \lto \ZZ$$
such that $\varphi([a],[b]) = \lk(a,b)$ for any oriented simple closed curves $a$ and $b$ in $\Si\times I \sm F$ and $F$, respectively. The proof is similar to the classical case and makes use of the Mayer-Vietoris sequence for the decomposition $\Si \times I =(\Si\times I \sm F) \cup \Int(N(F))$.

Note that $H_1(W,\Si) \cong H_1(\Si\times I \sm F,\Si\times \{1\})$, and consider the reduced Mayer-Vietoris sequence for the decomposition  $\wh{X}_2= W_0 \cup W_1:$
\begin{equation*}\label{MVseq}
\xymatrix{\cdots \ar[r] & H_1(W_0\cap W_1)\ar[r]^{\alpha_{*}\qquad } & H_1(W_0,\Si_0)\oplus H_1(W_1, \Si_1)
\ar[r]^{\qquad \qquad \beta_*} & H_1(\wh{X}_2)\ar[r]& \cdots.}
\end{equation*}

The intersection $W_0\cap W_1$ is a copy of the surface $\wt{F}$, which is connected since $F$ is non-orientable.  Thus, $\be_*$ is surjective, and any matrix representative of $\al_*$ is a presentation matrix for $H_1(\wt{X}_2).$ To complete the  proof, take a basis for $H_1(W_0\cap W_1) \cong H_1(\wt{F})$ and write out a representative for $\al_*$. This is a standard argument similar to the proof of \cite[Theorem 9.3]{Lickorish}; the details are left to the reader.
\end{proof}

\begin{corollary}
Suppose $F_1,F_2\subset \Si\times I$ are two (not necessarily $S^*$-equivalent) spanning surfaces for a link $L$. Let $A_1,A_2$ be the corresponding mock Seifert matrices, respectively. Then $|\det(A_1)|=|\det(A_2)|$.  
\end{corollary}

\begin{remark}
Recall that $\det(A)$ can be interpreted as the order of $H_1(X_2)$. For knots, $H_1(X_2)$ is finite of odd order, thus $\det(A)$ is always an odd integer. For links, $\det(A)=0$ if and only if $H_1(X_2)$ has infinite order. In any case, $\det(A)$ can be computed as the determinant of the coloring matrix, cf. Proposition 3.1 \cite{Boden-Karimi-2022}. 
It follows that $\det(A)$ is a welded invariant of the link $L$. In fact, the same is true of any link invariant derived from $H_1(X_2)$.
\end{remark}


\subsection{Realizability of mock Seifert matrices} \label{S3-4}
In this section, we prove a realizability result for mock Seifert matrices, showing that the
converse to \Cref{GL-matrix} (and \Cref{odd-det}) is true. 

\begin{theorem} \label{thm_realize}
A square integral matrix $A$ is the mock Seifert matrix for some $\ZZ/2$ null-homologous knot $K$ in a thickened surface if and only if $\det A$ is odd.  
\end{theorem}

\begin{proof}
\Cref{GL-matrix} and \Cref{odd-det} prove the statement in one direction, and we explain the other implication.

Let $A$ be a square integral matrix with $\det A$ odd. We will construct a $\ZZ/2$ null-homologous knot $K$ and a spanning surface $F$ whose mock Seifert matrix equals $A$.  Since the set of mock Seifert matrices is invariant under unimodular congruence, and since $A$ is assumed to have non-singular mod 2 reduction $\wbar{A}$,  it is enough to prove this under the assumption that $\wbar{A}$ is  equal to $H^{\oplus g}$ or $I_n$.

The proof in the two cases is similar, so assume $\wbar{A}=I_n$. (For the other case, see Theorem 3.7 of \cite{Boden-Chrisman-Gaudreau-2020}.)

The proof is by induction on $n$. In case $n=1$, suppose $A =[a]$ with $a$ odd. Consider the $(2,a)$ torus knot $K$, and let $F$ be an annulus with $|a|$ half twists, where the twists are right-handed if $a>0$ and left-handed if $a<0.$ Then $F$ is a spanning surface for $K$, and its mock Seifert matrix is easily seen to be $[a].$

Now suppose $n=2$ and consider the integral matrix 
$$A = \begin{bmatrix} a & b \\ c & d \end{bmatrix},$$ 
where $a,d$ are odd and $b,c$ are even. We show how to realize $A$ from a virtual spanning surface with two bands as in \Cref{base-case}.

\begin{figure}[ht]
\centering
\includegraphics[scale=1.20]{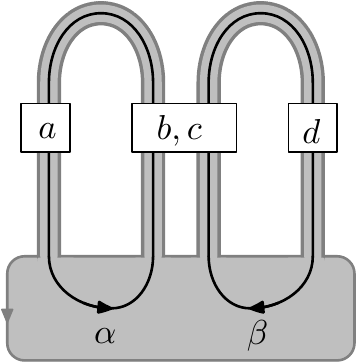} 
\caption{\small A virtual spanning surface with $n=2$.}
\label{base-case}
\end{figure}

Let $\al,\be$ be the cores of the bands in \Cref{base-case}, and
insert $a,d$ half-twists into the bands, using right-handed twists if $a$ or $d$ is positive and left-handed twists if $a$ or $d$ is negative. (The fact that $a,d$ are odd guarantees that $\partial F$ is connected.)    Also insert band crossings into the bands, as in \Cref{band-twists}, to ensure that $\lk(\tau \al,\be) = b$ and $\lk(\tau \be,\al) = c$. (Here, it is important that $b,c$ are even.) 

The core curves $\{\al,\be\}$ give a basis for $H_1(F)$, and with respect to this basis the linking form $\cL_F$ is represented by the matrix $A$.

\begin{figure}[ht]
\centering
\includegraphics[scale=0.90]{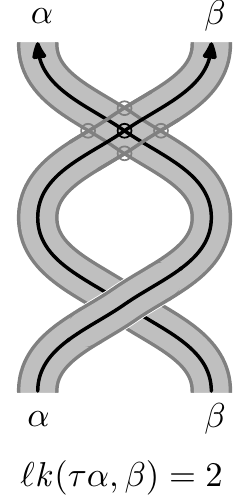} \quad
\includegraphics[scale=0.90]{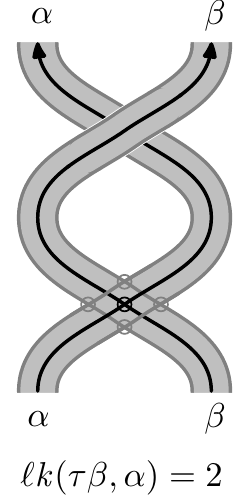} \quad
\includegraphics[scale=0.90]{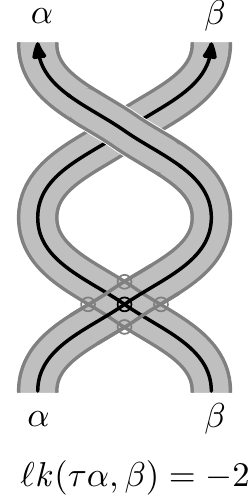} \quad
\includegraphics[scale=0.90]{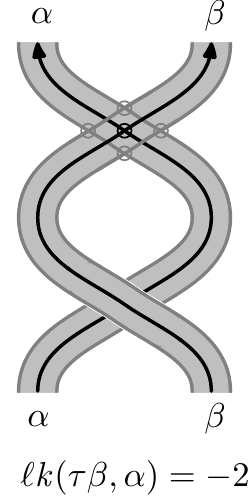} 
\caption{\small The four types of band crossings.}
\label{band-twists}
\end{figure}

The proof proceeds by induction. Let $A$ be a square integral $(n+1) \times (n+1)$ matrix whose mod 2 reduction is $I_{n+1}$. Let $A_1$ be the $n \times n$ matrix obtained by removing the last row and column from $A$ and let $b$ be the entry of $A$ in the last row and column. Thus,  $$A=
\left[\begin{array}{cc}
A_1 & * \\ 
* & b  
\end{array}\right].$$
By induction, we have a virtual spanning surface $F_1$ realizing $A_1$, and \Cref{band-induct} depicts a virtual band surface obtained by adding one band to $F_1$. In the figure, bands are represented by lines. The large box labelled $A_1$ indicates what can be arranged by induction. The curves $\{\al_1,\ldots, \al_{n+1} \}$  give a  basis for $H_1(F)$.

\begin{figure}[ht]
\centering
\includegraphics[scale=1.20]{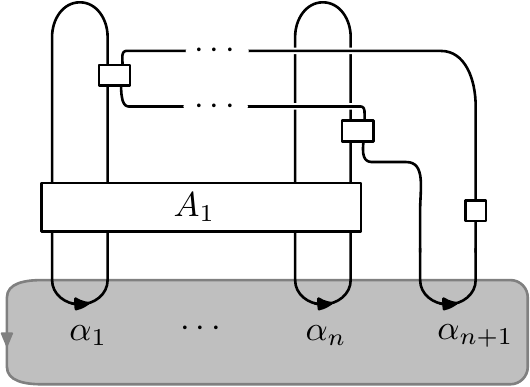} 
\caption{\small A virtual spanning surface for the inductive step.}
\label{band-induct}
\end{figure}

The empty boxes in  \Cref{band-induct} indicate twists and band crossings that need to be inserted. For instance, we insert $|b|$ half twists into the last band, and the twists are right-handed if $b>0$ and left-handed if $b<0.$ As well, we insert band crossings between the last band and each of the other bands.   The band crossings are selected as in  \Cref{band-twists} so that  $\lk(\tau \al_{n+1},\al_{i})$ and $\lk(\tau \al_{i},\al_{n+1})$ agree with the entries in the last row and column of $A$, respectively. It is now straightforward to verify that the surface $F$ has mock Seifert matrix $A$ with respect to the basis $\{\al_1,\ldots, \al_{n+1} \}$ for $H_1(F)$.
\end{proof}

\begin{remark} \label{rem-Euler}
For any virtual spanning surface $F$, one can compute the Euler number $e(F)$ by taking a longitude $K'$ that misses $F$ and applying \cref{eqn-Euler}. If $F$ is orientable, then $e(F)=0$. In that case, $b_1(F)=2g$ is even. On the other hand, if $F$ is not orientable, let $n=b_1(F)$. A direct computations reveals that $e(F) \equiv 2n \text{ (mod 4)}.$ It follows that $e(F)$ is always even. This also shows that if $e(F)=0$ then $b_1(F)$ is even.
\end{remark}
 
\section{Mock Alexander polynomial and Levine-Tristram signatures}  \label{section-4}
In this section, we introduce invariants derived from mock Seifert matrices, including the mock Alexander polynomial and mock Levine-Tristram signatures. Both are invariants of knots in thickened surfaces and depend a choice of spanning surface. However, the spanning surface is not assumed to be orientable. We show that the invariants depend only on the $S^*$-equivalence class of the spanning surface. We will also see that they take a very special form on slice knots. The mock Alexander polynomial satisfies a Fox-Milnor condition and the mock Levine-Tristram signatures vanish when the knot is slice. 

\subsection{Mock Alexander polynomial} \label{S4-1}
In this section, we define the mock Alexander polynomial associated to a $\ZZ/2$ null-homologous knot in a thickened surface. The polynomial depends on a choice of spanning surface, which is not assumed to be orientable. We show that the mock Alexander polynomial depends only on the $S^*$-equivalence of the spanning surface.

We begin with a purely algebraic definition of Alexander polynomial associated to a square integral matrix.
\begin{definition} \label{def-Alex}
If $A$ is a square integral matrix, then its Alexander polynomial is defined to be $\De_A(t) = \det(tA-A^\tr).$ 
\end{definition}

Since a matrix and its transpose have the same determinant, we see that, for any $n \times n$ integral matrix $A$, its Alexander polynomial satisfies
$$\De_A(t^{-1}) = \det(t^{-1}A-A^\tr)= t^{-n} \det(A^\tr - tA)=(-t)^{-n} \De_A(t).$$
In particular, in case $n$ is even, we have $\De_A(t)= t^n  \De_A(t^{-1}).$

The Alexander polynomial is an invariant of the congruence class of $A$. If $P$ is unimodular, then  
$$\De_{P^\tr \! A P}(t) = \det\left(t(P^\tr \! A P)-(P^\tr \! A P)^\tr \right)=\det \left(P^\tr (tA- A^\tr)P \right) = \De_A(t),$$
since $\det(P) =\det(P^\tr) = \pm 1.$

\begin{definition} \label{def-metabolic}
An integral square matrix $A$ is said to be \textit{\textbf{metabolic}} if it is unimodular congruent to a matrix in block form
$$\begin{bmatrix} \textbf{0} & B \\ C & D \end{bmatrix},$$ where $B,C,D$ are square matrices.
If $A$ is metabolic, then it must be of size $2n \times 2n$.
\end{definition}
In \cite{Levine-1969-a}, such matrices are called \textit{null-cobordant}.  

\begin{lemma}\label{lemma-FM}
If $A$ is metabolic, then $\De_A(t) = (-t)^n f(t) f(t^{-1})$ for some polynomial $f(t)\in\ZZ[t]$.
\end{lemma}

\begin{proof}
Since $A$ is metabolic, there is a unimodular matrix $P$ such that
$$P^\tr AP=\begin{bmatrix} \textbf{0} & B \\ C & D \end{bmatrix}.$$
Since the Alexander polynomial is invariant under congruence, we have
\begin{eqnarray*}
\De_A(t)=\De_{P^\tr AP}(t)=\det\left(t (P^\tr AP) - (P^\tr A P)^\tr\right)=\det\left(\begin{bmatrix} \textbf{0} & tB-C^\tr \\ tC-B^\tr & tD-D^\tr \end{bmatrix}\right). 
\end{eqnarray*}
Let $f(t)=\det(tB-C^\tr)$. Then 
$$(-t)^n f(t^{-1})=(-t)^n \det(t^{-1}B-C^\tr) = \det(tC^\tr-B) = \det(tC-B^\tr),$$ and the result now follows.
\end{proof}

Let $K$ be a $\ZZ/2$ null-homologous knot in $\Si \times I$ and $F$ a spanning surface for $K$. 
\begin{definition}
The \textit{\textbf{mock Alexander polynomial}} of $(K,F)$ is denoted $\De_{K,F}(t)$ and defined by
$$\De_{K,F}(t)=\det(tA-A^\tr),$$
where $A$ is the mock Seifert matrix associated to the Gordon-Litherland form $\cL_F$.
\end{definition}

The polynomial $\De_{K,F}(t)$ is defined as a Laurent polynomial, i.e., an element in $\La=\ZZ[t,t^{-1},(1-t)^{-1}]$, well-defined up to multiplication by units in $\La$. If $\De_1(t), \De_2(t) \in \La$ are polynomials satisfying $\De_1(t)=\pm t^k(1-t)^\ell \De_2(t)$ for some $k,\ell \in \ZZ$, then we write  $\De_1(t)\doteq \De_2(t)$.

\begin{proposition} \label{prop-Alex-inv}
If $K \subset \Si \times I$ is a $\ZZ/2$ null-homologous knot with spanning surface $F$, then the Alexander polynomial $\De_{K,F}(t)$ depends only on the $S^*$-equivalence class of $F$. In other words, if $F'$ is another spanning surface for $K$ that is $S^*$-equivalent to $F$, then $\De_{K,F}(t) \doteq \De_{K,F'}(t).$
\end{proposition}

\begin{proof}
Let $A$ be the mock Seifert matrix of $\cL_F$ with respect to a basis for $H_1(F).$
If $F$ and $F'$ are $S^*$-equivalent, then $F'$ is obtained from $F$ by a finite sequence of moves of the three types. 

The first move is ambient isotopy. Under an isotopy, $A$ changes by congruence, and the mock Alexander polynomial $\De_{K,F}(t)$ is invariant under congruence.

The second move is to attach (or remove) a tube. Suppose $F'$ is a surface obtained from $F$ by adding a tube to $F$. Then the mock Seifert matrix for $\cL_{F'}$ is  
$$A'=\begin{bmatrix} A& * & \textbf{0} \\ * & * & 1 \\ \textbf{0} & 1 & 0 \end{bmatrix}.$$ 
Therefore, 
\begin{eqnarray*}
\De_{K,F'}(t)&=&\det(tA'-(A')^\tr)=\det\begin{bmatrix} tA-A^\tr& * & \textbf{0} \\ * & * & t-1 \\ \textbf{0} & t-1 & 0 \end{bmatrix},\\
&=&\pm(t-1)^2\det(tA-A^\tr)\doteq \De_{K,F}(t). 
\end{eqnarray*}

The third move is to add (or remove) a half-twisted band. Suppose $F'$ is the surface obtained by adding a half-twisted band to $F$. Then the mock Seifert matrix for $\cL_{F'}$ is 
$$A'=\begin{bmatrix} A & \textbf{0}  \\ \textbf{0} & \pm 1 \end{bmatrix},$$
where the sign of the new entry is determined by whether the twist is right or left-handed. Therefore, 
$$\De_{K,F'}(t)=(t-1)\De_{K,F}(t)\doteq \De_{K,F}(t). $$
\end{proof}

In the next result, we show that the mock Alexander polynomial satisfies the Fox-Milnor condition on virtually slice knots.

\begin{theorem}\label{Fox-Milnor}
If $K\subset \Si \times I$ is virtually slice and $F$ is a spanning surface for $K$, then there exists an integral polynomial $f(t)$ such that $\De_{K,F}(t)\doteq f(t)f(t^{-1})$.
\end{theorem}

\begin{proof}
Choose a spanning surface $F$ with $e(F)=0.$ By \Cref{rem-Euler}, $b_1(F)$ is even, say $b_1(F)=2n$. Suppose $A$ is a mock Seifert matrix for $\cL_F$. Then $A$ has size $2n \times 2n$.  Since $K$ is virtually slice, Theorem 3.2 of \cite{Boden-Karimi-2021} applies to show that $A$ is metabolic, and \Cref{lemma-FM} applies to give the desired result.
\end{proof}

A local knot is a knot in $\Si\times I$ that is contained in some 3-ball. If $K$ is a local knot, then it admits a spanning surface $F$ that is also contained in a 3-ball. Let $\{ \al_1, \ldots, \al_{n}\}$ be a basis for $H_1(F)$, and $A$ be its mock Seifert matrix. Then 
$$\cL_F(\tau\al_i,\al_j)-\cL_F(\tau\al_j,\al_i)=p_*(\al_i)\cdot p_*(\al_j)=0,$$
where $p\co\Si\times I\to \Si$ is the projection. It follows that $A=A^\tr$, and $\De_{K,F}(t)=(t-1)^n\det(A)\doteq \det(A)$. Thus $\De_{K,F}(t)$ is just a constant polynomial.

This is also true for classical knots. In fact, for any knot $K$ in $S^2 \times I$, it is not difficult to show that its mock Alexander polynomial is a constant equal to $\det(K)$.

Suppose $K \subset \Si \times I$ is a $\ZZ/2$ null-homologous knot and $\xi \in H_2(\Si\times I,K;\ZZ/2)$. The \textit{\textbf{crosscap number}} is denoted $C_\xi(K)$ and defined to be the minimum $b_1(F)$ over all nonorientable spanning surfaces $F$ for $K$ with $[F]=\xi$ (cf., Definition 1.7  \cite{Boden-Chrisman-Karimi-2021}). 

For the next result, the \textit{\textbf{span}} of a mock Alexander polynomial $\De_{K,F}(t)$ is denoted $\text{span}(\De_{K,F}(t))$ and defined to be the difference between the highest and lowest degree terms in $\De_{K,F}(t)$ after factoring out $(t-1)^\ell$. 

For example, the mock Alexander polynomial
$$\De_{K,F}(t)=(t-1)(3t^2 +2t+3)\doteq 3t^2 +2t+3$$ has $\text{span}(\De_{K,F}(t))=2.$
  
\begin{proposition} \label{prop-CC}
Let $K \subset \Si \times I$ be a $\ZZ/2$ null-homologous knot and suppose $\xi \in H_2(\Si\times I,K;\ZZ/2).$
If $F$ is a spanning surface for $K$ with $[F]=\xi$,
then $\text{span}(\De_{K,F}(t)) \leq C_\xi(K).$
\end{proposition}
\begin{proof}
Let $F$ be a spanning surface for $K$ and suppose $n=b_1(F)$.
Then the mock Seifert matrix for $\cL_F$ is an integral matrix $A$ of size $n \times n,$
and $\De_{K,F}(t) = \det(tA - A^\tr)$ has span at most $n$.
\end{proof}

\begin{figure}[ht]
\centering
\includegraphics[scale=0.90]{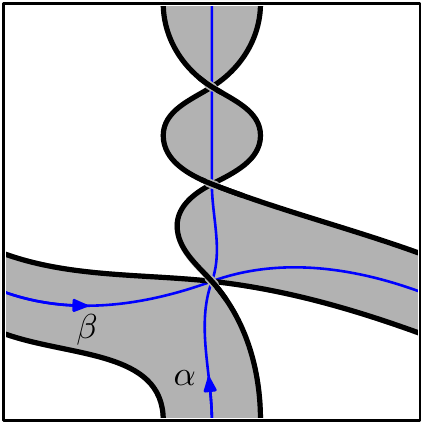} \qquad
\includegraphics[scale=0.90]{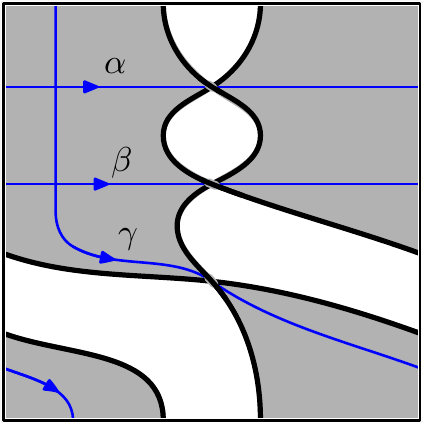}
\caption{\small The alternating virtual knot 3.7 with checkerboard surfaces $F$ (left) and $F'$ (right), and curves giving bases for $H_1(F)$ and $H_1(F')$.} \label{3-7-a}
\end{figure}

\begin{example}
The alternating virtual knot $K=3.7$ is shown in \Cref{3-7-a}, along with checkerboard surfaces $F$ and $F'$. For the first surface, using the basis $\{\al, \be\}$ for $H_1(F)$, we compute that $\cL_F$ has mock Seifert matrix
$$A = \begin{bmatrix} -3 & -2 \\ 0 & -1\end{bmatrix}.$$ 
Hence $$\De_{K,F}(t) =3t^2-2t+3.$$

For the second surface, using the basis $\{ \al, \be, \ga\}$ for $H_1(F')$, we compute that  $\cL_{F'}$ has mock Seifert matrix
$$ A' = \begin{bmatrix} 1 & 0& -1 \\ 0 & 1 & -1 \\ 1 & 1& 1 \end{bmatrix}.$$
Thus
\begin{equation*}
\begin{split}
\De_{K,F'}(t) 
&=(t-1)(3t^2+2t+3) \\ 
& \doteq 3t^2+2t+3.
\end{split}
\end{equation*}
Since $\text{span}(\De_{K,F}(t))=\text{span}(\De_{K,F'}(t))=2$, \Cref{prop-CC} applies to show that this knot has crosscap number $C_\xi(K) \geq 2$. \hfill $\Diamond$
\end{example}

If $K \subset \Si \times I$ is $\ZZ$ null-homologous and $\al \in H_2(\Si \times I, K;\ZZ/2)$, then the \textit{\textbf{Seifert genus}} is denoted $g_\al(K)$ and defined to be the minimum genus $g(F)$ over all orientable spanning surfaces $F$ for $K$ with $[F] = \al$. The proof of the next result is similar to \Cref{prop-CC} and left to the reader.
  
\begin{proposition} \label{prop-genus}
Let $K \subset \Si \times I$ be a $\ZZ$ null-homologous knot and suppose $\al \in H_2(\Si\times I,K;\ZZ/2).$
If $F$ is an orientable spanning surface for $K$ with $[F]=\al$,
then $\text{span}(\De_{K,F}(t)) \leq 2g_\al(K).$
\end{proposition}

We end this section by mentioning an interesting problem, which is to find necessary and sufficient conditions for a  polynomial $\De(t)$ to occur as $\De_A(t)= \det(tA - A^\tr)$ for a mock Seifert matrix $A$. There are three necessary conditions: (i) the mod 2 reduction of $\De(t)$ is $(t+1)^n$, (ii)  $\De(t)  = t^n \De(t^{-1})$, and (iii) $\De(1)=0$ if $n$ is odd and $\De(1)$ is a square divisible by $2^n$ if $n$ is even. We do not know whether these conditions are sufficient.

The proofs of (i) and (ii) are elementary, and we sketch the proof of (iii). Assume that $A$ has mod 2 reduction $I_n$ and set $B = (A - A^\tr)/2$. Note that $B$ is integral since the off-diagonal entries of $A$ are all even. Writing $\De_A(t) = \det((t-1)A + (A - A^\tr))$ and evaluating at $t=1$, we have  $\De_A(1) = \det(A - A^\tr) = \det(2B) = 2^n \det(B).$  Since $B$ is skew-symmetric, $\det(B) = \det(B^\tr)= (-1)^n \det(B).$ If $n$ is odd, then $\det(B)=0,$  whereas if $n$ is even, a formula due to Cayley \cite{Cayley-1849} shows that $\det(B) = Pf(B)^2$, where $Pf(B)$ is the Pfaffian of $B.$  The eigenvalues of $B$ are purely imaginary and come in conjugate pairs. Thus, it follows that $\De_A(1)$ is either zero (if $n$ is odd) or a square divisible by $2^n$ (if $n$ is even).
 
\subsection{Mock Levine-Tristram signatures} \label{S4-2}

In this section we introduce the mock Levine-Tristram signature invariants for $\ZZ/2$ null-homologous knots in thickened surfaces. These signatures depend on a choice of spanning surface, which is not assumed to be orientable. We show that the mock Levine-Tristram signatures depend only on the $S^*$-equivalence of the spanning surface.

We begin with a few algebraic observations. Suppose $A$ is a mock Seifert matrix. Let $\om \in \CC$ be a complex number with $|\om|=1$ and set  $H_\om = (1-\om)A + (1-\wbar{\om})A^\tr$. Then  $H_\om$ is a Hermitian matrix and has a well-defined signature.

Recall from \Cref{def-Alex} that $\De_A(t) = \det(A -tA^\tr)$.

\begin{lemma}\label{lemma-nonsig}
Let $\om \in S^1\sm \{1\}$.  If $\De_A(\om) \neq 0$, then $H_\om$ is non-singular.
\end{lemma}
\begin{proof}
Clearly $(\wbar{\om}-1)(\om A-A^\tr)=(1-\om)A+(1-\wbar{\om})A^\tr =H_\om$. Thus, if $\om \neq 1,$ then $\De_A(\om)\neq 0$ implies $\det(H_\om)\neq 0$, and so $H_\om$ is non-singular.
\end{proof}

\begin{lemma}\label{lemma-LT}
Let  $\om\in S^1\sm \{1\}$.
If $\De_A(\om) \neq 0$ and $A$ is metabolic, then $\sig(H_\om)=0.$
\end{lemma}

\begin{proof}
By \Cref{lemma-nonsig}, $H_\om$ is non-singular.   
Since $A$ is metabolic, we have a unimodular matrix $P$ such that 
$$P^\tr AP=\begin{bmatrix} \textbf{0} & B  \\ C & D  \end{bmatrix}.$$ 
Any non-singular Hermitian form that vanishes on
a half-dimensional subspace has signature zero, and that completes the proof.
\end{proof}

Let $K\subset \Si \times I$ be a $\ZZ/2$ null-homologous knot and $F$ a spanning surface for $K$. In the following, $e(F)$ is the Euler number of $F$, which is given by  \cref{eqn-Euler}.  By \Cref{rem-Euler}, $e(F)$ is always even.

\begin{definition}
For $\om \in S^1\sm \{1\}$, the \textit{\textbf{mock Levine-Tristram signature}} of $(K,F)$ is denoted $\si_{K,F}(\om)$ and defined by
$$\si_{K,F}(\om)=\sig(H_\om)+\tfrac{1}{2}e(F),$$
where $A$ is the mock Seifert matrix associated to the Gordon-Litherland form $\cL_F$ and $H_\om=(1-\om)A+(1-\wbar{\om})A^\tr$.
\end{definition}

\begin{proposition}
The Levine-Tristram signature $\si_{K,F}(\om)$ of a knot depends only on the $S^*$-equivalence class of $F$. In other words, if $F'$ is another spanning surface for $K$ that is $S^*$-equivalent to $F$, then $\si_{K,F}(\om) = \si_{K,F'}(\om)$.
\end{proposition}

\begin{proof}
Let $A$ be the mock Seifert matrix of $\cL_F$ with respect to a basis for $H_1(F).$ If $F$ and $F'$ are $S^*$-equivalent, then $F'$ is obtained from $F$ by a finite sequence of moves of the three types. The first move is ambient isotopy.
Under an isotopy, $A$ changes by congruence, and the Levine-Tristram signature $\si_{K,F}(\om)$ is invariant under congruence.

The second move is to attach (or remove) a tube. Suppose $F'$ is a surface obtained from $F$ by adding a tube to $F$. Then for $\om=e^{i\th}$, we can write the mock Seifert matrix of $\cL_{F'}$ as in the proof of \Cref{prop-Alex-inv} and deduce that
$$H'_\om=\begin{bmatrix} (1-\om)A+(1-\wbar{\om})A^\tr & * & \textbf{0} \\ * & * & 2-2\cos\th \\ \textbf{0} & 2-2\cos\th & 0 \end{bmatrix}.$$
Since $\om\neq 1$, $2-2\cos\th\neq 0$, and using this number and by a number of simultaneous row and column operations  we can reduce $H'_\om$ to the following matrix without changing its signature. 

$$\begin{bmatrix} (1-\om)A+(1-\wbar{\om})A^\tr & \textbf{0} & \textbf{0} \\ \textbf{0} & 0 & 2-2\cos\th \\ \textbf{0} & 2-2\cos\th & 0 \end{bmatrix}.$$

The signature of this matrix equals $\sig\left((1-\om)A+(1-\wbar{\om})A^\tr\right)$. Since adding a tube does not change the Euler number $e(F)$, it follows that $\si_{K,F}(\om)=\si_{K,F'}(\om)$. 

The third move is to add (or remove) a half-twisted band. Suppose $F'$ is a surface, obtained from $F$ by adding a half-twisted band. Then the mock Seifert matrix of $\cL_{F'}$ is 
$$A'=\begin{bmatrix} A & \textbf{0}  \\ \textbf{0} & \pm 1 \end{bmatrix},$$
where the sign of the new entry is determined by whether the twist is right or left-handed.  
Therefore,
$$H'_\om = (1-\om)A'+(1-\wbar{\om})(A')^\tr=\begin{bmatrix} (1-\om)A+(1-\wbar{\om})A^\tr & \textbf{0}  \\ \textbf{0} & \pm(2-2\cos\th) \end{bmatrix}.$$
This matrix has signature $\sig((1-\om)A+(1-\wbar{\om})A^\tr)\pm 1$. 
Thus $ \sig(H'_\om) =\sig(H_\om)\pm 1,$ whereas $e(F') = e(F)\mp 2.$
It follows that $\si_{K,F'}(\om)=\si_{K,F'}(\om).$
\end{proof}

\begin{theorem} \label{LT-vanish}
Let $\om\in S^1\sm \{1\}$. If $K\subset \Si \times I$ is a virtually slice knot with spanning surface $F$ and $\De_{K,F}(\om)\neq 0$, then $\si_{K,F}(\om)=0$.
\end{theorem}

\begin{proof}
Since $\si_{K,F}(\om)$ depends only on the $S^*$-equivalence class of $F$, by adding half-twisted bands, we can arrange that $e(F)=0$. Let $A$ be the mock Seifert matrix for the Gordon-Litherland form $\cL_F$.
Since $K$ is virtually slice, Theorem 3.2 of \cite{Boden-Karimi-2021} applies to show that $A$ is metabolic. The result now follows from \Cref{lemma-LT}.
\end{proof}

\begin{definition}
For $\om=e^{i\th}\in S^1$, the average signature of $K$, denoted $\si_{K,F}^{\text{avg}}(\om)$ and defined as 
$$\si_{K,F}^{\text{avg}}(\om)=\tfrac{1}{2}\left( \lim_{\eta\to \th^{-}}\sig(H(e^{i\eta}))+
\lim_{\eta\to \th^{+}}\sig(H(e^{i\eta}))+e(F)\right).$$
\end{definition}

\begin{example}
The knot $K=6.78358$ is almost classical and therefore admits an orientable spanning surface. Its Seifert matrices $V^+$ and $V^-$ are computed in \cite{Boden-Chrisman-Gaudreau-2020}, and they determine the mock Seifert matrix by the simple formula $A = V^+ + V^-.$ In particular, we have:
$$V^+=\begin{bmatrix} 0&0&1&0 \\ -1&1&0&0 \\ -1&1&0&1 \\ 0&0&0&1 \end{bmatrix}, \quad 
V^- =\begin{bmatrix} 0&-1&0&0 \\ 0&1&1&0 \\ 0&0&0&0 \\ 0&0&1&1 \end{bmatrix}, \quad  
A =\begin{bmatrix} 0 &-1& 1& 0 \\ -1&2&1&0 \\ -1&1&0&1 \\ 0&0&1&2\end{bmatrix}.$$  
Thus 
$$\De_{K,F}(t) = 5t^4 - 4t^3 - 2t^2 - 4t + 5 \doteq 5t^2 + 6t + 5.$$ 
Notice that $\De_{K,F}(t)$ has roots $(-3 \pm 4 i)/5$ on the unit circle. The Levine-Tristram signature $\si_{K,F}(\om)$ jumps from $0$ to $2$ at the roots. \hfill $\Diamond$
\end{example}

\section{Unoriented algebraic concordance} \label{section-5}
In this section, we construct the concordance group $\mG^\ZZ$ of mock Seifert matrices and define an analogue of the Levine homomorphism, which is a surjection $\la\co \vC \to \mG^\ZZ$. 
We define algebraic sliceness for long knots in thickened surfaces.
We also show that $\mG^\ZZ$ is large by showing that it contains an infinite linearly independent subset. In fact, we will see that $\mG^\ZZ$ contains a subgroup isomorphic to $\ZZ^\infty \oplus (\ZZ/2)^\infty \oplus (\ZZ/4)^\infty.$

\subsection{Admissible matrices} \label{S5-1}
In this section,  we define the set of admissible matrices, which is a subset of the set of mock Seifert matrices.

A square integral matrix is said to be \textit{\textbf{even}} if each of its diagonal entries is an even integer. Otherwise, it is said to be \textit{\textbf{odd}}. The sets of even and odd matrices are invariant under unimodular congruence. In particular, a mock Seifert matrix $A$ is even if and only if its mod 2 reduction is unimodular congruent to $H^{\oplus g}$; and it is odd if and only if its mod 2 reduction is unimodular  congruent to $I_n$. 

For any mock Seifert matrix $A$, the entries of $A+A^\tr$ and $A-A^\tr$ are all even integers. This is evidently true if $A$ reduces mod 2 to $H^{\oplus g}$ or $I_n$, and it is preserved under unimodular congruence. Thus, if $A$ is a mock Seifert matrix, then $(A+A^\tr)/2$ and $(A-A^\tr)/2$ are both integral matrices, and of course $(A+A^\tr)/2$ is symmetric and $(A-A^\tr)/2$ is skew-symmetric. 

Next, we introduce the set of admissible matrices. 
\begin{definition}
A square integral $n \times n$ matrix $A$ is said to be \textit{\textbf{admissible}} if it is the mock Seifert matrix for some $\ZZ/2$ null-homologous knot in a thickened surface with spanning surface $F$ such that $e(F)=0$.
\end{definition} 

The set of admissible matrices is invariant under unimodular congruence. By \Cref{rem-Euler}, it follows that any admissible matrix must have an even number of rows.  

Let $A$ be an admissible matrix, and assume $A$ is odd of size $n \times n,$ where $n$ is even. Then it represents the linking form $\cL_F$ for a non-orientable surface $F$ with $b_1(F)=2m$, where $n=2m$. By the classification of compact surfaces with connected boundary, $F$ can be described as a disk with $2m$ bands attached as in \Cref{vbp_base}. The bands are attached in pairs with their feet alternating. Each band except the $(2m-1)$-st is attached with an even number of half twists, and the $(2m-1)$-st band has an odd number of half twists. Let $\al_1, \ldots, \al_{2m}$ be the cores of the bands, and let $A$ be the mock Seifert matrix of $\cL_F$ with respect to the basis $\al_1, \ldots, \al_{2m}$ for $H_1(F)$.  Then $A$ has mod 2 reduction equal to $H^{\oplus (m-1)} \oplus K,$ where  
\begin{equation} \label{eqn-matrices}
H= \begin{bmatrix}0 & 1 \\ 1 & 0\end{bmatrix} \quad \text{and} \quad K= \begin{bmatrix}1 & 1 \\ 1 & 0\end{bmatrix}.
\end{equation}
Writing $A=(a_{ij})$ for $1 \leq i,j\leq n,$ a direct calculation (see \Cref{euler-contribution}) shows that $F$ has Euler number equal to $a_{nn}$. In particular, an integral matrix of size $n \times n$ and odd type is admissible if and only if (i) $n$ is even, (ii) it is unimodular congruent to a matrix $A$ whose mod 2 reduction is $H^{\oplus (m-1)} \oplus K$, and has $a_{nn}=0.$

\begin{figure}[ht]
\centering
\includegraphics[scale=.72]{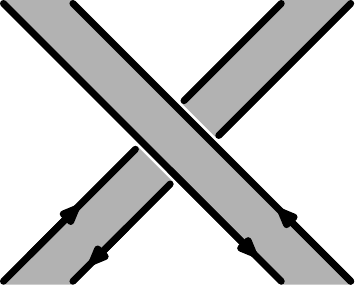} \hspace{.1cm}
\includegraphics[scale=.72]{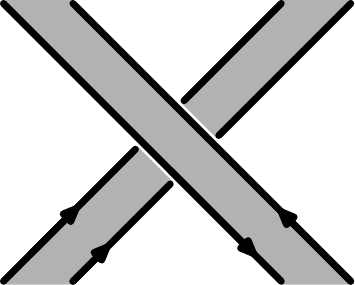} \hspace{.1cm}
\includegraphics[scale=.72]{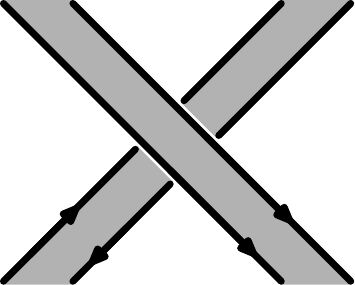} \hspace{.1cm}
\includegraphics[scale=.72]{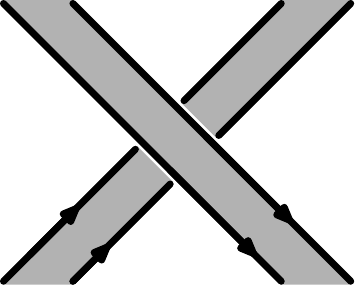} \hspace{.1cm}
\includegraphics[scale=.72]{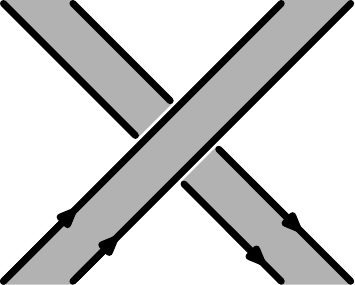} 
\caption{\small In the three pictures on the left, the contribution to the Euler number is zero. In the two pictures on the right, the contribution to the Euler number is $4$ and $-4$, respectively.}
\label{euler-contribution}
\end{figure}

\begin{figure}[ht]
\centering
\includegraphics[scale=1.30]{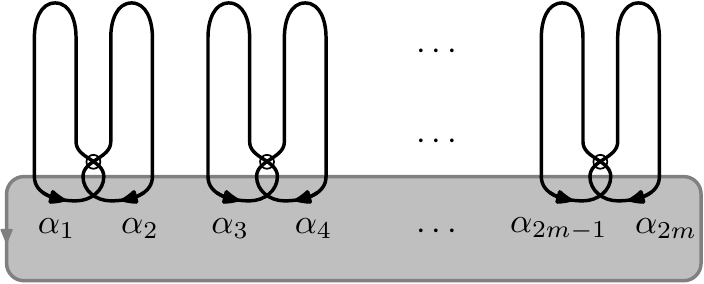} 
\caption{\small Each band except the $(2m-1)$-st has an even number of half twists.}
\label{vbp_base}
\end{figure}

The next result summarizes the conditions satisfied by admissible mock Seifert matrices. 

\begin{proposition} \label{prop:GL}
Let $A$ be an $n \times n$ integral matrix with $n$ even and $\det A$ odd. Then $A$ is a mock Seifert matrix.  If $A$ is an even matrix, then it is automatically admissible (since it is the mock Seifert matrix for an orientable surface).

If $A$ is an odd matrix, then up to unimodular congruence, its mod 2 reduction is equal to $H^{\oplus (m-1)} \oplus K$, where $H,K$ are the $2 \times 2$ matrices in \cref{eqn-matrices} above. In that case, if $a_{nn}=0$, then $A$ is admissible.

Conversely every admissible matrix $A$ is an $n \times n$ integral matrix with $n$ even and $\det A$ odd that is unimodular congruent to one of the above two types.
\end{proposition}

\subsection{The concordance group of matrices} \label{S5-2}
In this section, we recall the notion of concordance for matrices and construct the concordance group $\mG^\ZZ$ of admissible mock Seifert matrices, following Levine \cite{Levine-1969-a}.

Recall from \Cref{def-metabolic} that an integral square matrix $A$ of size $2n \times 2n$ is said to be \textit{\textbf{metabolic}}  (or \textit{null-concordant}) if it is unimodular congruent to a matrix  in block form
$$\begin{bmatrix} \textbf{0} & B \\ C & D \end{bmatrix},$$ where $B,C,D$ are square $n \times n$ matrices.

Given two square matrices $A$ and $B$, the block sum is defined to be the matrix
$$A \oplus B =\begin{bmatrix} A & \textbf{0} \\  \textbf{0} & B \end{bmatrix}.$$ 
The matrices $A$ and $B$ are said to be \textit{\textbf{concordant}} if the block sum $A \oplus (-B)$ is metabolic. (Such matrices are called \textit{cobordant} in \cite{Levine-1969-a}.)

Concordance defines a relation on admissible matrices where $A \sim B$ if $A \oplus (-B)$ is null-concordant.  It is easy to see that $\sim$ is reflexive and symmetric, and a straightforward Witt cancellation argument shows that $\sim$ is transitive. Let $A,B,C$ be admissible matrices with $A \sim B$ and $B \sim C.$  Then $(A \oplus -B) \oplus (B \oplus -C)$ and $N=B \oplus -B$ are null-concordant. Since $N$ has $\det(N) = \det(B)^2 \neq 0$, Lemma 1 of \cite{Levine-1969-a} applies to show that $A \oplus -C$ is also null-concordant. It follows that $A \sim C.$

The set of admissible matrices is closed under the operation $A \oplus B$ of block sum. Therefore,  the set of concordance classes of admissible matrices forms an abelian group under block sum. This group is denoted $\mG^\ZZ$ and called the \textit{\textbf{concordance group of mock Seifert matrices}}.

\subsection{Mock algebraically slice knots} \label{S5-3}
In this section, we define mock algebraic sliceness for long knots in thickened surface. We will also prove a lemma that is useful in showing the Levine homomorphism is well-defined.

\begin{definition}
A $\ZZ/2$ null-homologous long knot $(K,q)$ is said to be \textit{\textbf{mock algebraically slice}} if its mock Seifert matrix $A$ is null-concordant. Here $A$ is the matrix representative of $\cL_F$ for $F$ a preferred spanning surface for $K$ with $e(F)=0$.  
\end{definition}

\begin{theorem} Let $(K,q)$ be a $\ZZ/2$ null-homologous long knot in $\Si \times I$. If $(K,q)$ is virtually slice, then it is mock algebraically slice.
\end{theorem}
\begin{proof}
Assume $K$ is virtually slice and $F$ is a spanning surface for $K$ with $e(F)=0$.  Then the proof of Theorem 3.2 in \cite{Boden-Karimi-2021} implies that there is a basis $\{ \al_1, \ldots, \al_{2n}\}$ for $H_1(F)$ such that $\cL_F(\tau \al_i,\al_j)=0$ for $1\leq i,j\leq n$. If $A$ is the mock Seifert matrix for $\cL_F$, this shows that $A$ is null-concordant and the result follows.
\end{proof}

Recall that a classical knot is said to be \textit{\textbf{algebraically slice}} if it admits a Seifert matrix which is null-concordant. If $K$ is a classical knot and $V$ is a Seifert matrix for $K$, then the matrix $A=V+V^\tr$ obtained by symmetrizing $V$ is a mock Seifert matrix for $K$. Clearly if $V$ is null-concordant, then so is $V+V^\tr$. Therefore, any classical knot that is algebraically slice is necessarily mock algebraically slice. The converse is not true, and one can easily find examples of classical knots that are mock algebraically slice but not algebraically slice.

\begin{lemma} \label{lem-conc}
Let $(K_i,q_i)$ be a $\ZZ/2$ null-homologous long knot in $\Si_i \times I$ for $i=0,1$, and let $F_i$ be a preferred spanning surface for $K_i$ with $e(F_i)=0$. Let $A_i$ be the associated mock Seifert matrix for $K_i$, which is admissible. If $K_0$ and $K_1$ are virtually concordant as long knots, then the matrices $A_0$ and $A_1$ are concordant.
\end{lemma}
\begin{proof} 
Consider the connected sum $-K_0 \# K_1$, with spanning surface $F_0 \#_b F_1$ given by the boundary connect sum of $F_0$ and $F_1$. Then $e(F_0 \#_b F_1)=e(F_0)+e(F_1)=0.$ A direct computation shows that $-K_0 \# K_1$ has mock Seifert matrix equal to $-A_0 \oplus A_1$. By \Cref{slice-concordance}, $-K_0 \# K_1$ is virtually slice, and the proof of Theorem 3.2 in \cite{Boden-Karimi-2021} applies to show that $-A_0 \oplus A_1$ is null-concordant. It follows that $A_0$ is concordant to $A_1$.  
\end{proof}
\subsection{Levine homomorphism} \label{S5-4}

In this section, we will define an analogue of the Levine homomorphism $\la \co \vC \lto \mG^\ZZ$. This map is most naturally defined on the subgroup 
\begin{equation}\label{eqn-cc}
\vC_2 = \{ [(K,q)] \mid K \text{ is a $\ZZ/2$ null-homologous long knot}\} \subset \vC
\end{equation}
consisting of concordance classes of $\ZZ/2$ null-homologous long knots, where it sends the long knot to its associated mock Sefert matrix in $\mG^\ZZ.$  It is extended to a map on $\vC$ by precomposing with the surjection $\varphi_2 \co \vC \lto \vC_2$ induced by parity projection. This will be explained in \Cref{section-7}. For now, we just show that there is a well-defined surjection defined on $\vC_2$.

Let $(K,q)$ be a $\ZZ/2$ null-homologous long knot in $\Si \times I$, and let $F$ be a preferred spanning surface for $(K,q)$ with $e(F)=0.$ (This can be arranged by adding half-twisted bands.) Let $A$ be the mock Seifert matrix of $\cL_F$ in some basis for $H_1(F)$. Notice that $A$ is an integral matrix of
size $2n \times 2n$, since $b_1(F)=2n$ is necessarily even, cf. \Cref{rem-Euler}, and that $A$ is admissible. The map $\la$ is then defined by sending  the concordance class of $(K,q)$ to the concordance class of $A$.

To see that this map is well-defined, apply \Cref{lem-conc} to show that the concordance class of $A$ in $\mG^\ZZ$ depends only on the concordance class of $(K,q).$ Thus we obtain a well-defined map $\la \co \vC_2 \lto \mG^\ZZ$, which one can easily see is a homomorphism.
A further application of \Cref{prop:GL} shows that $\la$ is a surjection.

\subsection{The concordance group of rational matrices} \label{S5-5}
For classical knots, Levine proved that the map $\cG^\ZZ \to \cG^\QQ$ is injective, where $\cG^\ZZ$ and $\cG^\QQ$ denote the integral and rational algebraic concordance group of Seifert matrices, respectively. In this section, we will prove an analogous result for the concordance groups of mock Seifert matrices.

We begin by introducing the concordance group of $2n \times 2n$ matrices over an arbitrary field $\FF$. A non-singular $2n \times 2n$ matrix $A$ with entries in $\FF$ is said  to be \textit{\textbf{$\FF$-metabolic}} (or \textit{$\FF$-null-concordant}) if the corresponding form vanishes on an $n$-dimensional subspace of $\FF^{2n}$. Equivalently, $A$ is metabolic if there is a non-singular matrix $P$ over $\FF$ such that 
$$PAP^\tr = \begin{bmatrix} \textbf{0} & B \\ C & D \end{bmatrix},$$ where $B,C,D$ are square $n \times n$ matrices over $\FF$.

Two non-singular square matrices $A$ and $B$ over $\FF$ are said to be \textit{\textbf{$\FF$-concordant}} if  $A \oplus (-B)$ is $\FF$-metabolic. In that case, we write $A \sim_\FF B$. It is not difficult to verify that $\sim_\FF$ determines an equivalence relation on non-singular $2n \times 2n$ matrices over $\FF.$ (See Theorem 3.4.4 in \cite{Livingston-Naik-2016} for a detailed proof.\footnote{According to Definition 3.4.1 in \cite{Livingston-Naik-2016}, in constructing $\cG^\FF$, they only consider matrices such that $A+A^\tr$ and $A-A^\tr$ are non-singular. For instance, in Lemma 3.4.5 in \cite{Livingston-Naik-2016}, they prove Witt cancellation under the assumption that $N-N^\tr$ is non-singular. But Lemma 1 from \cite{Levine-1969-a} indicates that Witt cancellation holds more generally.}
) 
Under block sum $\oplus$, the set of $\FF$-concordance classes of non-singular $2n \times 2n$ matrices forms an abelian group denoted $\mG^\FF$ and called the \textit{\textbf{concordance group of matrices over $\FF$}}. In the case $\FF=\QQ$, this group is denoted $\mG^\QQ$ and called the \textit{\textbf{concordance group of rational matrices}}.

The next result is the analogue for concordance groups of mock Seifert matrices of \cite[Lemma 8]{Levine-1969-a}.
 
\begin{theorem}
The natural map $\mG^\ZZ \to \mG^\QQ$ induced by inclusion is an injective homomorphism.
\end{theorem}

For a proof that carries over to our setting, see Theorem 3.4.12 in \cite{Livingston-Naik-2016} (cf., \cite{Levine-1969-b}). Note that our set of admissible matrices is different from  \cite{Livingston-Naik-2016} and \cite{Levine-1969-b}. Levine considers matrices $A$ satisfying $\det((A-A^\tr)(A+A^\tr))\neq 0$, and he shows that every concordance class admits a non-singular representative, see \cite[Lemma 8]{Levine-1969-b} and \cite[Theorem 3.4.10]{Livingston-Naik-2016}. By  \Cref{odd-det}, every admissible matrix has $\det A$ odd, therefore $A$ is automatically non-singular over $\QQ$.

Levine also defined an isomorphism $\cG^\QQ  \to \cG_\QQ$, where $\cG_\QQ$ denotes the group of Witt classes of isometric structures, see \cite[Theorem 8]{Levine-1969-b}. He showed that a class in $\cG_\QQ$ is trivial if and only if it is trivial in $\cG_F$ for $F=\RR$ and $F = \QQ_p$ for all primes $p,$ where $\QQ_p$ is the field of $p$-adic rationals. This leads to a complete list of invariants for $\cG^\QQ$ which Levine used to prove that $\cG^\QQ$  is isomorphic to $\ZZ^\infty \oplus (\ZZ/2)^\infty \oplus (\ZZ/4)^\infty$. Stolzfus proved that $\cG^\ZZ$  is also isomorphic to $\ZZ^\infty \oplus (\ZZ/2)^\infty \oplus (\ZZ/4)^\infty$ \cite{Stoltzfus}.

If $A$ is an admissible mock Seifert matrix, then since $\det A$ is odd, $A$ is invertible over $\QQ$. Set $Q = (A + A^\tr)/2$ and $S = (A^\tr)^{-1} A$. Then $Q$ is a bilinear pairing and $S$ satisfies  
\begin{eqnarray*}
S^\tr (A + A^\tr) S &=&  A^\tr A^{-1} (A + A^\tr) (A^\tr)^{-1} A,\\
                &=&  A^\tr A^{-1} A (A^\tr)^{-1} A + A^\tr A^{-1} A^\tr (A^\tr)^{-1} A,\\
                &=&  A + A^\tr.
\end{eqnarray*}
It follows that $Q(Sx, Sy) = Q(x,y)$ for all $x,y \in \QQ^n$. Thus $S$ is an isometry for the bilinear pairing $Q$.

We can now consider isometric structures $(V,Q,S)$, where $Q$ is a symmetric bilinear form on $V$ and $S$ is an isometry of $V$. Let $\mG_\QQ$ be the group of concordance classes of isometric structures. The characteristic polynomial is $\De_S(t)= \det(tI-S).$ Given $A$, we can associate $Q = (A+A^\tr)/2$ and $S=(A^\tr)^{-1} A$, and this defines a homomorphism $\mG^\QQ \to \mG_\QQ.$ If $\det(S+I) \neq 0,$ then we can recover $A$ from $(Q,S)$ by the formula $$A=2Q S (S+I)^{-1}.$$ For this to work, we  require 
$$\det(S+I)=\det((A^\tr)^{-1} A +I)=\det((A^\tr)^{-1}) \det(A +A^\tr) \neq 0.$$
This is equivalent to the condition that $\De_S(-1)\neq 0$. This condition is not always satisfied in our setting.

\subsection{The concordance group of matrices is large}  \label{S5-6}

In this section, we show that $\mG^\ZZ$ contains a subgroup isomorphic to $\ZZ^\infty \oplus (\ZZ/2)^\infty \oplus (\ZZ/4)^\infty$.

We begin by showing that $\mG^\ZZ$ contains an infinite linearly independent set. For classical knots, this was proved by Levine \cite[Proposition 6]{Levine-1969-a} and independently by Milnor \cite[Theorem in \S 5]{Milnor-1968}. Our proof makes use of the mock Levine-Tristram signatures.

\begin{proposition} \label{prop-inf-lin-ind}
The group $\mG^\ZZ$ contains an infinite linearly independent set.
\end{proposition}

\begin{proof}
Consider the matrix 
$$A_k=\begin{bmatrix} k&1\\-1&k \end{bmatrix}.$$ 
If $k$ is even, then $A_k$ is admissible.  For $\omega \in S^1$, let $\si_k(\om)=\sig((1-\om)A_k+(1-\wbar{\om})A_k^\tr)$. Notice that $\si_k(-1)=2$ if $k>0.$

Let $S_k$ be obtained from $S^1$ by removing the points $\left(1-\frac{2}{k^2+1}\right) \pm i\left(\frac{2k}{k^2+1}\right)$, and let $N_k$ be the component of $S_k$ containing $-1$. Then $\si_k$ is constant on the components of $S_k$. If $k>0,$ then $\si_k(-1)=2$, and it follows that $\si_k(\om)=2$ for all $\om \in N_k$. Note also that $\si_k(\om)=0$ for all $\om \in S_k\sm N_k$. 

Assume now that $k$ is even. Observe that $N_2\subset N_4 \subset \cdots \subset N_k \subset N_{k+2} \subset \cdots$, and the inclusions are proper. We claim that the set $\{A_{2j}\}_{j=1}^{\infty}$ is linearly independent in $\mG^\ZZ$. To prove the claim, it is enough to show that every finite subset $\{A_{2j}\}_{j=1}^m$ is linearly independent.

Suppose to the contrary that $A=\bigoplus_{j=1}^m \la_{2j} A_{2j}$ is null-concordant. Then $A$ is non-singular, and $\sig((1-\om)A+(1-\wbar{\om})A^\tr)=0$ for all $\om \in \bigcap_{j=1}^m S_{2j}$. We can further assume that $\la_{2m}\neq 0$. If $\om\in N_{2m}\sm N_{2m-2}$, then $\sig((1-\om)A+(1-\wbar{\om})A^\tr)=2\la_{2m}$, which is a contradiction.   
\end{proof}

\Cref{prop-inf-lin-ind} implies that $\mG^\ZZ$ contains a copy of $\ZZ^\infty.$ The next result shows that $\mG^\ZZ$ also contains copies of $(\ZZ/2)^\infty$ and $(\ZZ/4)^\infty$. 
 
\begin{proposition} \label{prop-Z2-Z4}
The group $\mG^\ZZ$ contains copies of $(\ZZ/2)^\infty$ and $(\ZZ/4)^\infty$.
\end{proposition}

\begin{proof} 
The argument uses the facts that (i) the set of primes with $p \equiv 1$ mod 4 is infinite,
and (ii) the set of primes with $p \equiv 3$ mod 4 is also infinite.
We use (i) to show that $\mG^\ZZ$ contains a copy of $(\ZZ/2)^\infty$
and (ii) to show it contains a copy of $(\ZZ/4)^\infty$.

Let $\text{Diag}(p,q)$ denote the diagonal $2 \times 2$ matrix with entries $p,q.$ For instance,  if $p,q$ are distinct primes with $p \equiv q\equiv 1 \text{ (mod $4$)}$, then the matrix  $A=\text{Diag}(-p,q)$ is admissible and has order 2 in $\mG^\ZZ.$ To see that $A \oplus A$ is null-concordant, observe that when $p \equiv 1 \text{ (mod $4$)}$,  $-1$ is a square in $\FF_p$, the finite field of order $p$. Therefore, $\text{Diag}(p,p)$ is Witt equivalent to $\text{Diag}(p,-p)$, which is easily seen to be null-concordant (see \cite[Appendix A.2, p.180]{Livingston-Naik-2016}).

Let $\{p_i \mid 1\leq i<\infty \}$ be an infinite sequence of increasing primes with $p_i =4k_i+1.$ Consider the family  of $2 \times 2$ matrices $\{A(i) \mid 1\leq i<\infty \}$ given by setting $A(i) = \text{Diag}(-p_i,p_{i+1})$. Notice that $A(i)$ is an admissible matrix of order two in $\mG^\ZZ.$

We claim that the matrices $\{A(i) \mid 1\leq i<\infty\}$ generate a copy of $(\ZZ/2)^\infty$ in $\mG^\ZZ.$ To prove the claim, consider the element $A = A(1) \oplus A(2)\oplus \cdots \oplus A(n)$. Using a telescoping argument together with the observation that $\text{Diag}(-p_i,p_i)$ is null-concordant, one can show that $A$ is concordant to $\text{Diag}(-p_1,p_{n+1}),$ which is nontrivial of order two in $\mG^\ZZ.$

Actually, we need to prove the following more general statement. Suppose  $1 \leq i_1 < i_2 < \cdots < i_n$ and consider the element $A = A(i_1) \oplus A(i_2) \oplus \cdots \oplus A(i_n)$. We claim that $A$ is a nontrivial element of order two in $\mG^\ZZ$. The largest prime factor of $\det(A)$ is $p_{n+1}$. While it divides $\det(A)$, its square $(p_{n+1})^2$ does not. Therefore, $\det(A)$ is not a perfect square, and it follows that $A$ is not null-concordant. Thus the matrices $\{A(i) \mid 1\leq i<\infty\}$ generate a copy of $(\ZZ/2)^\infty$  in $\mG^\ZZ$.

The same method can be used to show that $\mG^\ZZ$ contains a copy of $(\ZZ/4)^\infty.$  We sketch the argument  using (ii).

First, observe that if $p,q$ are distinct  primes and with $p \equiv q\equiv 3 \text{ (mod $4$)}$, then $A=\text{Diag}(-p,q)$ has order four in $\mG^\ZZ.$ For this, we again refer to \cite[Appendix A.2, p.180]{Livingston-Naik-2016}.

Now let  $\{p_i \mid 1\leq i<\infty \}$ be an infinite sequence of increasing primes with $p_i \equiv 3 \text{ (mod $4$)}$ for  $1\leq i<\infty.$ Set $A(i) = \text{Diag}(-p_i,p_{i+1})$ and consider the family of $2 \times 2$ matrices $\{A(i) \mid 1\leq i<\infty \}$. Each $A(i)$ is an admissible matrix of order four in $\mG^\ZZ.$

Then in a similar way, one can prove that  the matrices $\{A(i) \mid 1\leq i<\infty\}$ generate a copy of $(\ZZ/4)^\infty$ contained in $\mG^\ZZ$. The details are left to the reader.
\end{proof}

The natural map from $\cG^\ZZ$, the classical algebraic concordance group to $\mG^\ZZ$, the concordance group of mock Seifert matrices is induced by symmetrization. Namely, the map  $\cG^\ZZ \to \mG^\ZZ$ is given by $V \mapsto V+V^\tr,$ where $V$ is a Seifert matrix. It is an interesting problem to describe the image of this map. 

For instance, if $K$ is a classical knot with signature $\si(K)=0$, then we conjecture that it maps to a torsion element in $\mG^\ZZ$. (It is clear that any knot with $\si(K)\neq 0$ has infinite order in $\mG^\ZZ.$) We claim that the image of $\cG^\ZZ \to \mG^\ZZ$ contains a copy of $(\ZZ/2)^\infty$. Does it also contain a copy of $(\ZZ/4)^\infty?$

Example 12.2.14 in \cite{Kawauchi-1990} describes a family of knots $K_n, n\geq 1$ which all have order two in $\cC$. They generate a copy of $(\ZZ/2)^\infty$ in $\cC$. Since $\cC \to \vC$ is injective, they also generate a copy of $(\ZZ/2)^\infty$ in $\vC.$ We claim that they generate a copy of $(\ZZ/2)^\infty$ in $\mG^\ZZ$ as well. 

Consider the knot  $K=K_{n_1} \# K_{n_2} \# \cdots \#K_{n_\ell},$ where $1 \leq n_1 < n_2 < \cdots <n_\ell.$ Then $\det(K) = (4n_1^2+1)(4n_2^2+1)\cdots (4n_\ell^2+1)$. If $n_i$ are chosen so that $n_i^2+1$ is prime, then it follows that $[K]$ is nontrivial in $\mG^\ZZ$. Thus the family of knots $K_n, n\geq 1$ generates a copy of $(\ZZ/2)^\infty$ in the image of $\cG^\ZZ \to \mG^\ZZ.$

\section{Parity projection and the concordance group of virtual knots} \label{section-7}
In this section, we introduce parity and use it to describe certain natural subgroups of the concordance group $\vC$ of virtual knots. For a full account on parity, see  \cite{Manturov-2010, Ilyutko-Manturov-Nikonov-2011, Nikonov-2016}.

A parity is a family of functions $\{f_D\}$, one for each virtual knot diagram $D$. Each is a function on the set $\{c\}$ of classical crossings of $D$ and taking values in the set $\{0,1\}$. Crossings with $f_D(c)=0$ are called \textit{\textbf{even}}, and crossings with $f_D(c)=1$ are called \textit{\textbf{odd}}.
The family of functions $\{f_D\}$ is required to satisfy the following axioms.

\begin{itemize}[leftmargin=*]  
\item[$\circ$] Under a detour move, the parity of each crossing is unchanged. 
\item[$\circ$] Under a Reidemeister move, the parity of every crossing that is not involved in the move is unchanged. 
\item[$\circ$] If $c$ is a crossing that is eliminated in a Reidemeister I move, then $c$ is even.  
\item[$\circ$] If $c_1,c_2$ are crossings that are eliminated in a Reidemeister II move, then $f_D(c_1)=f_D(c_2)$. 
\item[$\circ$] If $c_1,c_2, c_3$ are three crossings involved in a Reidemeister III move relating $D$ and $D'$, then $f_D(c_i)=f_{D'}(c_i)$ for $i=1,2,3.$ In addition, either $\{c_1,c_2,c_3\}$ are all even, or they are all odd, or exactly two of them are odd.  
\end{itemize}

Given a parity $f$, there is a map called \textit{parity projection}, which is denoted $P_f$ 
and defined by replacing odd crossings with virtual crossings. On the level of virtual knot diagrams, every odd crossing is virtualized, i.e.,
$\begin{tikzpicture}
\draw[line width=1.50pt] (.2,-.18) -- (-.2,.18);
\draw[line width=1.50pt](.2,.18) -- (.05,.05);
\draw[line width=1.50pt](-.2,-.18) -- (-.05,-.05);
\draw[line width=0.75pt,->] (.4,.00)--(0.9,.00);
\draw[line width=1.50pt] (1.5,-.18) -- (1.1,.18);
\draw[line width=1.50pt](1.5,.18) -- (1.1,-.18);
\draw[line width=0.75pt] (1.3,0) circle (3pt); 
\end{tikzpicture}$.

In \cite{Manturov-2010}, Manturov proved that if $D$ and $D'$ are related by Reidemeister moves, then so are $P_f(D)$ and $P_f(D')$. It follows that parity projection induces a well-defined map on the level of virtual knots. Evidently,  $P_f(D) = D$ if and only if every crossing of $D$ is even. Further, under repeated application, the diagram $P^{k}_f(D)$ eventually contains only even crossings, at which point it stabilizes. The map $P_f^\infty = \lim_{k \to\infty}P_f^k$ is called \textit{stable projection}.  

There are many different parity functions. The simplest examples are the \textit{Gaussian parities}, defined next.

The \textit{mod $n$ Gaussian parity} is denoted $f_n$ and defined for $n>1$ by  
$$f_n(c) = \begin{cases} 0 & \text{if $\ind(c)\equiv 0$ (mod $n$),}\\
1 &  \text{otherwise.}
\end{cases}$$
It is easy to check that $f_n$ satisfies the parity axioms.
Let $P_n$ denote parity projection with respect to $f_n$.
Then $P_n(D)=D$ if and only if $D$ is mod $n$ almost classical.
Here, recall that a diagram $D$ is said to be \textit{mod $n$ almost classical}
if $\ind(c) \equiv 0 \text{ (mod $n$)}$ for all $c$ in $D$ (see \Cref{S1-3} for the formula for $\ind(c)$). Note that a diagram $D$ is mod $2$ almost classical if and only if it is checkerboard colorable.

The \textit{total Gaussian parity} is denoted $f_0$ and defined by  
$$f_0(c) = \begin{cases} 0 & \text{if $\ind(c)= 0$,}\\
1 &  \text{otherwise.}
\end{cases}$$
It is easy to see that $f_0$ satisfies the parity axioms. Let $P_0$ denote parity projection with respect to $f_0$. Then $P_0(D)=D$ if and only if $D$ is almost classical.

The parity projection maps $P_n$ and $P_0$ can also be defined for long virtual knots. By Theorem 5.11 in \cite{Boden-Chrisman-Gaudreau-2020}, they respect concordance. Specifically, if $K_0$ and $K_1$ are concordant virtual knots, then $P_n(K_0)$ and $P_n(K_1)$ are concordant, as are $P_0(K_0)$ and $P_0(K_1)$. 

Let $P_n^\infty = \lim_{k\to\infty} P^k_n$ and $P_0^\infty$ denote the  stable parity projection maps with respect to the mod $n$ Gaussian parity $f_n$ and the total Gaussian parity $f_0$. Stable projection $P_2^\infty$ induces a surjective homomorphism 
$$\varphi_2 \co \vC \lto \vC_2,$$
where $\vC_2$ is the subgroup of $\vC$ defined in \eqref{eqn-cc} consisting of concordance classes of $\ZZ/2$ null-homologous long knots. Alternatively, $\vC_2$ can be viewed as the subgroup of concordance classes of checkerboard colorable long virtual knots. The odd writhe vanishes on checkerboard colorable knots and is a concordance invariant. Therefore, $\vC_2$ is a proper subgroup of $\vC$. 

Likewise, stable projection $P_0^\infty$ induces a surjective homomorphism 
$$\varphi_0 \co \vC \lto \vC_0,$$
where  $$\vC_0= \{[(K,q)] \mid K \text{ is a $\ZZ$ null-homologous long knot} \}$$ is the subgroup of $\vC$ consisting of concordance classes of $\ZZ$ null-homologous long knots. Alternatively, $\vC_0$ can be viewed as the subgroup of concordance classes of almost classical long virtual knots. Every classical knot is almost classical, and every almost classical knot is checkerboard colorable. These observations imply that there are inclusion maps, each of which is proper:
$$\cC \subsetneq \vC_0 \subsetneq \vC_2 \subsetneq \vC.$$

Consider the surjections $\varphi_0 \co \vC \to \vC_0$ and $\varphi_2 \co \vC \to \vC_2$ induced by stable projection. Let $N_0 = \ker(\varphi_0)$ and $N_2 = \ker(\varphi_2)$ be the normal subgroups given by the kernels of $\varphi_0$ and $\varphi_2,$ respectively. They lead to two short exact sequences:
\begin{equation}\label{eq-ses}
\begin{split}
&1 \to N_0 \lto \vC \lto \vC_0 \to 1, \\
&1 \to N_2 \lto \vC \lto \vC_2 \to 1.
\end{split}
\end{equation}

The inclusion maps $\vC_0 \hookrightarrow \vC$ and $\vC_2 \hookrightarrow \vC$ give splittings of the sequences in \eqref{eq-ses}, and it follows that $\vC$ can be written as a semidirect product in two ways:
\begin{equation}\label{equation-semidirect}
\vC= N_0 \rtimes \vC_0 = N_2 \rtimes \vC_2,
\end{equation}
where $\vC_0$ acts on $N_0$  by conjugation, and $\vC_2$ acts on $N_2$ by conjugation.

The subgroup $N_0$ consists of concordance classes of virtual knots whose image $[P^\infty_0(K)]$ is virtually slice. In particular, any virtual knot $K$ containing only chords with nonzero index has its concordance  class lying in $N_0$. More generally, given a virtual long knot $K$, the connected sum $K \# \left(-P^\infty_0(K)\right)$ lies in $N_0$, and every element of $N_0$ is concordant to a virtual long knot of this form.

Turaev's polynomial invariants $u_\pm(K)$ are concordance invariants (see \cite{Turaev-2008-A}), and the coefficients of the nonzero degree terms vanish on $\vC_0$. These invariants imply that $N_0$ is infinitely generated.

Analogously, the subgroup $N_2$ consists of concordance classes of virtual knots whose image $[P^\infty_2(K)]$ is virtually slice.
In particular, any virtual knot $K$ containing only chords with odd index has its concordance  class lying in $N_2$. More generally, given a virtual long knot $K$, the connected sum $K \# \left(-P^\infty_2(K)\right)$ lies in $N_2$, and every element of $N_2$ is concordant to a virtual knot of this form.

Since $u_\pm(K)$ are concordance invariants, and since the coefficients of their odd degree terms vanish on $\vC_2$, it follows that $N_2$ is infinitely generated. 

Although the concordance group $\vC$ of virtual knot is not abelian \cite{Chrisman-2020},
the operation of connected sum $K_0 \# K_1$ is commutative if either $K_0$ or $K_1$ is classical.
Therefore the concordance group $\cC$ of classical knots lies in the center of $\vC$ and there is a central extension
$$\cC \lto \vC \lto  Q$$
with quotient group $Q$.  
In fact, there are central extensions
$$\cC \lto \vC_0 \lto  Q_0 \quad \text{and} \quad \cC \lto \vC_2 \lto  Q_2\quad $$
with quotient groups $Q_0$ and $Q_2$. 

Turaev's polynomials $u_\pm(K)$ apply to show that $Q$ and $Q_2$ are both infinitely generated. \Cref{prop-inf-lin-ind} applies to show that $Q_0$ is also infinitely generated.
It would be interesting to know more about the structure of the groups $Q, Q_0$ and $Q_2$.

An algorithm for computing mock Seifert matrices has recently been developed by Damian Lin. The input is an alternating Gauss code, and it determines the Tait graphs, Gordon-Litherland forms, and mock Seifert matrices for both checkerboard colorings, see \cite{Lin-algo}.

\subsection*{Acknowledgements}
The first author was partially funded by the Natural Sciences and Engineering Research Council of Canada.
The authors would like to thank  Chuck Livingston, Andy Nicas, and Will Rushworth for their valuable feedback.
They would also like to thank Zsuzsi Dancso, Damian Lin, and Tilda Wilkinson-Finch for their input.

\newpage
\bibliographystyle{alpha}
\newcommand{\etalchar}[1]{$^{#1}$}

\end{document}